\documentclass[arxiv]{default}

\title{Exponential sums over finite fields and the large sieve}
\def\titleHead{Exponential sums over finite fields and the large sieve}
\author{Corentin Perret-Gentil}
\address{Centre de recherches mathématiques, Montréal, Canada}
\email{corentin.perretgentil@gmail.com}

\newcommand{\FM}{\operatorname{FM}}

\date{March 2018. Revised August 2018.}

\begin{document}

\begin{abstract}
  By using a variant of the large sieve for Frobenius in compatible systems developed in \cite{KowLS06} and \cite{KowLargeSieve08}, we obtain zero-density estimates for arguments of $\ell$-adic trace functions over finite fields with values in some algebraic subsets of the cyclotomic integers, when the monodromy groups are known. This applies in particular to hyper-Kloosterman sums and general exponential sums considered by Katz.
\end{abstract}

\maketitle

\tableofcontents

\section{Introduction}

\subsection{Exponential sums and trace functions}\label{subsec:expSumsCycl}

We consider exponential sums over a finite field $\F_q$ of characteristic $p\ge 5$ such as:
\begin{enumerate}
\item\label{ex:Kln} Hyper-Kloosterman sums of rank $n\ge 2$ given by
\begin{equation}
    \label{eq:Kln}
    \Kl_{n,q}(a)=\frac{(-1)^{n-1}}{q^{(n-1)/2}}\sum_{\substack{x_1,\dots,x_n\in\F_q^\times\\x_1\cdots x_n=a}} e \left(\frac{\tr(x_1+\dots+x_n)}{p}\right),
  \end{equation}
  for $a\in\F_q^\times$, $\tr:\F_q\to\F_p$ the trace map, and $e(z)=\exp(2\pi i z)$ for any $z\in\C$. More generally, we also have hypergeometric sums as introduced in \cite[Chapter 8]{KatzESDE};
\item\label{ex:GenExpSums} General exponential sums of the form
  \begin{equation}
    \label{eq:expsumGen}
    \frac{-1}{q^{1/2}}\sum_{\substack{y\in\F_q\\ f(y),\, g(y),\, h(y)\neq\infty}} e\left(\frac{\tr(xf(y)+h(y))}{p}\right)\chi(g(y)),
  \end{equation}
  for $x\in\F_q$, $f,g,h\in\Q(X)$ rational functions and $\chi$ a character of $\F_q^\times$. For example, we have Birch sums $q^{-1/2}\sum_{y\in\F_q} e \left(\tr(xy+y^3)/p\right)$, cubic exponential sums studied in particular by Livné \cite{Liv87} and Katz;
\item\label{ex:familiesHE} Functions counting points on families of curves such as
  \begin{equation}
    \label{eq:familiesHE}
    \frac{q+1-|X_z(\F_q)|}{q^{1/2}} \hspace{0.2cm} (z\in\F_q, \ f(z)\neq 0),
  \end{equation}  
  where $X_z$ is the smooth projective model of the affine hyperelliptic curve $y^2=f(x)(x-z)$ over $\F_q$, for $f\in\Z[X]$ fixed squarefree of degree $2g\ge 2$.
\end{enumerate}

\subsubsection{Exponential sums as algebraic integers}\label{subsubsec:expalgint} Note that the three examples above all take values in the localization $\Z[\zeta_{4p}]_{q^{1/2}}$ (by the evaluation of quadratic Gauss sums), or less precisely in the cyclotomic field $\Q(\zeta_{4p})$.

It is an interesting question to investigate their properties as elements of these sets, as done by Fisher \cite{Fisher92,Fisher95} or recently by the author \cite{PGDistrTFCycl16} for the distribution of their reductions modulo a prime ideal and short sums thereof.

\subsubsection{Trace functions}

Examples \eqref{eq:Kln}--\eqref{eq:familiesHE} are specific incarnations of \textit{trace functions $t:\F_q\to\C$ arising from constructible middle-extension sheaves of $\overline\Q_\ell$-modules on $\P^1/\F_q$}, for $\ell$ a prime distinct from $p$, as constructed in particular by Deligne \cite{DelEC} and Katz \cite{KatzESDE}.

Very powerful tools are then available to study various aspects of these functions, such as Deligne's extension \cite{Del2} of the Riemann hypothesis for varieties over finite fields to weights of étale cohomology groups of such sheaves.

For example, Katz \cite{KatzGKM} obtained a ``vertical Sato--Tate law'' for the distribution of Kloosterman sums, through a general equidistribution theorem of Deligne, and similar results \cite{KatzESDE} for families of the type \eqref{eq:expsumGen} or \eqref{eq:familiesHE}.

\subsection{Zero-density estimates}\label{subsec:zerodensity}
The goal of the present article is to obtain general zero-density estimates of the form
\begin{equation}
  \label{eq:density}
  P \big(t(x)\in A\big):=\frac{|\{x\in\F_q : t(x)\in A\}|}{q}=o(1) \hspace{1cm} (q\to+\infty)
\end{equation}
where:
\begin{itemize}
\item $t: \F_q\to E$ is the trace function associated to a \textit{coherent family} of sheaves over $\F_q$ (Definition \ref{def:coherentFamily}), for $E$ a number field.
\item $A\subset E$ is an ``algebraic'' subset such as the set of $m$-powers ($m\ge 2$), the image of a polynomial, or more generally a set defined by a first-order formula in the language of rings.
\end{itemize}

This will apply in particular, with $E=\Q(\zeta_{4p})$, to Kloosterman sums \eqref{eq:Kln} and exponential sums of the form \eqref{eq:expsumGen}.

\subsubsection{Families of curves}

The large sieve for Frobenius in compatible systems was developed by Kowalski in \cite{KowLS06} and \cite{KowLargeSieve08} to obtain results of the type of Chavdarov \cite{Chav97} on zeta functions of families of curves, such as the probability that the numerator has Galois group as large as possible.

In the notations of Example \ref{ex:familiesHE} above, Kowalski gets for example (see \cite[Section 8.8]{KowLargeSieve08}) that
\begin{eqnarray*}
  P\Big(f(z)\neq 0,\ |X_z(\F_q)|\in\N^{\times 2}\Big)&:=&\frac{|\{z\in\F_q : f(z)\neq 0, \ |X_z(\F_q)|\in\N^{\times 2}\}|}{q}\\
&\ll&gq^{1-(4g^2+2g+4)^{-1}}\log{q},
\end{eqnarray*}
for $\N^{\times 2}$ the set of squares of integers.

The large sieve bound ultimately relies on estimates of exponential sums obtained through Deligne's generalization of the Riemann hypothesis over finite fields.

Note that in the setting above, we have $E=\Q$.

\subsubsection{Examples of results for Kloosterman sums}\label{subsubsec:Kln}
In the case of hyper-Kloosterman sums \eqref{eq:Kln} of rank $n\ge 2$, our main results are the following:
  
\begin{proposition}\label{prop:LSKlPower}
  Let $n\ge 2$ be an integer and $\varepsilon>0$. For $m\ge 2$ coprime to $p$, we have
  \[P\Big(\Kl_{n,q}(x)\in \Q(\zeta_{4p})^m\Big)\ll_{n, m,\varepsilon}\frac{p^\varepsilon\log{q}}{B_nq^{1/(2B_n)}}\to 0\]
  when $q=p^e\to+\infty$ is coprime to $n$ with $e\ge 16B_n$, where
  \begin{equation}
    \label{eq:Bn}
    B_n=
    \begin{cases}
      \frac{2n^2+n-1}{2}&:n\text{ odd}\\
      \frac{2n^2+3n+4}{4}&:n\text{ even},
    \end{cases}
  \end{equation}
and $\Q(\zeta_{4p})^m$ is the set of $m$th powers in $\Q(\zeta_{4p})$. The implied constant depends only on $n$, $m$ and $\varepsilon$.
\end{proposition}

More generally:
\begin{proposition}\label{prop:LSKlf}
  Let $n\ge 2$ be an integer and $\varepsilon>0$. For
  \begin{itemize}
  \item almost all\footnote{Throughout, this will mean ``for all but $o(h^d)$ such polynomials of height at most $h$, as $h\to+\infty$''.\label{fn:almostallpol}} monic polynomials $f\in\Z[X]$ of fixed degree $d\ge 2$, and
  \item all $f\in\Z[X]$ of degree $d\ge 2$ such that the Galois group of $f(X)-y\in\C(y)[X]$ is equal to $\Sf_d$,
  \end{itemize}
  we have
  \[P\Big(\Kl_{n,q}(x)\in f(\Q(\zeta_{4p}))\Big)\ll_{n, f,\varepsilon} \frac{p^\varepsilon\log{q}}{B_nq^{1/(2B_n)}}\to 0\]
  when $q=p^e\to+\infty$ is coprime to $n$ with $e\ge 16B_n$, for $B_n$ as in \eqref{eq:Bn}. The implied constant depends only on $n$, $f$ and $\varepsilon$.
\end{proposition}

\begin{remarks}\label{rem:exampleKl}~
  \begin{enumerate}
  \item The bounds are uniform in $p$, thanks to the determination of the finite monodromy groups in \cite{PGIntMonKS16}, over a field of characteristic $\ell\gg_n 1$.
  \item This can further be extended to definable subsets of $\Q(\zeta_{4p})$ (i.e. defined by a first-order formula in the language of rings), under some technical conditions (Proposition \ref{prop:definableSubsetsKl} later on).
  \item The same bounds hold for unnormalized Kloosterman sums.
  \item Under the general Riemann hypothesis (GRH) for the Dedekind zeta function of $\Q(\zeta_{4p})$, one may take $\varepsilon=0$ and $e\ge 4B_n+1$.
  \item By relying on the determination of the monodromy groups over $\overline\Q_\ell$ by Katz and the results of Larsen--Pink (see Section \ref{subsec:exploitMonodromyC}), instead of \cite{PGIntMonKS16}, these results would only hold when $p$ is fixed and $e\to+\infty$, with an implied constant depending on $p$.
  \end{enumerate}
\end{remarks}

\subsection{Strategy}

The general idea to obtain zero-density estimates of the type \eqref{eq:density} is the following: in the setting of Section \ref{subsec:zerodensity}, let $\Oc$ be the ring of integers of $E$. It turns out (by definition of a coherent family) that there exists a set $\Lambda$ of valuations of $\Oc$ (equivalently, of prime ideals) such for every $\lambda\in\Lambda$, the function $t:\F_q\to E$ coincides with the trace function $t_\lambda:\F_q\to \Oc_\lambda$ arising from a constructible middle-extension sheaf of $\Oc_\lambda$-modules on $\P^1/\F_p$. By reduction, we obtain a trace function $\tilde t_\lambda:\F_q\to\F_\lambda$, where $\F_\lambda$ is the residue field at $\lambda$.

Thus,
\[P \big(t(x)\in A\big)\le \frac{|\{x\in\F_q : \tilde t_\lambda(x)\in A_\lambda \ \forall\lambda\in\Lambda\}|}{q},\]
where $A_\lambda=(A\cap\Oc_\lambda)\pmod{\lambda}\subset\F_\lambda$. A variant of Kowalski's large sieve for Frobenius in compatible systems, handling sheaves of $\Oc_\lambda$-modules instead of sheaves of $\Z_\ell$-modules, can then be used to bound this quantity in terms of local densities in the sets $A_\lambda$.

\subsubsection{Technical tools}

More precisely, the first part of the approach requires:
\begin{itemize}
\item The construction by Deligne and Katz of examples of the form \ref{ex:Kln} and \ref{ex:GenExpSums} as trace functions of sheaves of $\Oc_\lambda$-modules.
\item Information on monodromy groups:
\begin{itemize}
\item When available, the determination of integral monodromy groups for a density one subset of the valuations, not depending on $p$.
\item Otherwise, results of Larsen and Pink \cite{LarsPink92,Lars95} to handle sheaves whose monodromy groups are known over $\overline\Q_\ell$ (e.g. by the works of Katz \cite{KatzGKM,KatzESDE}), but not over $\F_\lambda$.
\item For sheaves associated with exponential sums of the form \ref{ex:GenExpSums}, conditions and/or normalizations so that arithmetic and geometric monodromy groups coincide.
\end{itemize}
\end{itemize}

To compute local densities in the sets $A_\lambda$, we will need bounds on ``Gaussian sums'' (see Section \ref{sec:tracesrandommatrices}) over:
\begin{itemize}
  \item Linear algebraic groups over $\F_\lambda$; these follow either from Deligne's generalization of the Riemann hypothesis over finite fields \cite{Del2} and bounds of Katz on sums of Betti numbers \cite{KatzBetti01}, or from explicit computations of D.S. Kim for certain finite groups of Lie type.
  \item Subsets of $\F_\lambda$ such as powers (Bourgain and others, e.g. \cite{BouCh06}) or more generally definable subsets (Kowalski \cite{Kow07}, using the work of Chatzidakis--van der Dries--Macintyre \cite{CDM92}).
\end{itemize}

The implied constant in a bound of the form \eqref{eq:density} will depend on $p$ (forcing to fix $p$ and take $q=p^e$, $e\to+\infty$) when we rely on the results of Larsen--Pink, and will be absolute when more precise information about integral monodromy groups is available.\\

When we want results with absolute implied constants, we will also employ uniform estimates in Chebotarev's density theorem (e.g. \cite{May13}), since $E$ may depend on $p$.

\subsection{Organization of the paper}
In Section \ref{sec:largesieve}, we lay out the technical setup of trace functions of sheaves of $\Oc_\lambda$-modules over finite fields, define coherent families, and show that \eqref{eq:Kln} and \eqref{eq:familiesHE} arise from such families. Finally, we state a variant of the large sieve for Frobenius in compatible systems (Theorem \ref{thm:largeSieve}).

In Section \ref{sec:tracesrandommatrices}, we get results on the Gaussian sums mentioned above, which will be used to compute the local densities in the sieve.

In Section \ref{sec:zerodensity}, we apply the large sieve of Section \ref{sec:largesieve} to obtain bounds of the type \eqref{eq:density}, by using the estimates from Section \ref{sec:tracesrandommatrices} and uniform bounds in Chebotarev's density theorem.

In Section \ref{sec:examples}, we start by explaining how this leads to the results for Kloosterman sums given in Section \ref{subsubsec:Kln} above. Then, we work towards obtaining similar zero-density estimates for general exponential sums of the form \eqref{eq:expsumGen}, showing that coherent families can still be obtained through the results of Larsen and Pink (in particular with Theorem \ref{thm:monodromyLP}).

\begin{acknowledgements}
  The author would like to thank Emmanuel Kowalski and Richard Pink for helpful discussions, as well as the anonymous referees for very valuable comments. This work was partially supported by DFG-SNF lead agency program grant 200021L\_153647 and by the National Science Foundation under Grant No. 1440140, while the author was in residence at the Mathematical Sciences Research Institute in Berkeley, California, during the Spring semester of 2017. Some of the results also appeared in the author's PhD thesis.
\end{acknowledgements}

\section{The large sieve for Frobenius in compatible systems}\label{sec:largesieve}

We start by recalling the technical setup of trace functions over finite fields, before stating a version of the large sieve for Frobenius adapted to our needs.

Throughout this section, a number field $E$ with ring of integers $\Oc$ is fixed, as well as a finite field $\F_q$ of characteristic $p$.

\subsection{Trace functions over finite fields}

\subsubsection{Definitions}

Let $\lambda$ be an $\ell$-adic valuation corresponding to a prime ideal $\lf$ of $\Oc$, $E_\lambda$ and $\Oc_\lambda$ the completions, and $\F_\lambda\cong\Oc/\lf$ the residue field.

Let $A=\overline\Q_\ell$, $\Oc_\lambda$ or $\F_\lambda$. We recall that a constructible middle-extension sheaf of $A$-modules over $\P^1/\F_p$ (or \textit{sheaf of $A$-modules over $\F_p$} for simplicity) corresponds to a continuous $\ell$-adic Galois representation
\[\rho_\Fc: \pi_{1,p}:=\Gal\left(\F_p(T)^\sep/\F_p(T)\right)\to\GL(\Fc_{\overline\eta})\cong\GL_n(A),\]
for $\overline\eta$ a geometric generic point and $\F_p(T)^\sep$ the corresponding separable closure. The associated \textit{trace functions} are, for every finite extension $\F_q/\F_p$,
  \begin{eqnarray*}
    t_{\Fc}=t_{\Fc,q}: \F_{q}&\to&A\\
    x&\mapsto&\tr \left(\rho_\Fc(\Frob_{x,q}) \mid \Fc_{\overline\eta}^{I_x}\right),
  \end{eqnarray*}
  where\footnote{The set of conjugacy classes of a group $G$ will be denoted by $G^\sharp$.} $\Frob_{x,q}\in (D_x/I_x)^\sharp\cong\Gal(\overline\F_q/\F_{q})$ is the geometric Frobenius at $x\in\F_{q}$, for $I_x\normal D_x\le\pi_{1,p}$ the inertia (resp. decomposition) group at $x$. We will denote by $U_\Fc\subset\P^1$ the maximal open of lissité of $\Fc$.
  
We refer the reader to \cite[Chapter 2]{KatzGKM} for more details and references.

\subsubsection{Monodromy groups}\label{subsec:mono}

If $\Fc$ is a sheaf of $A$-modules over $\F_p$ as above, the \textit{arithmetic and geometric monodromy groups} of $\Fc$ are the groups
\[G_\geom(\Fc)=\rho_\Fc\big(\pi_{1,p}^\geom\big)\le G_\arith(\Fc)=\rho_\Fc(\pi_{1,p})\le\GL_n(A)\]
if $A$ is discrete, and
\[G_\geom(\Fc)=\overline{\rho_\Fc\big(\pi_{1,p}^\geom\big)}\le G_\arith(\Fc)=\overline{\rho_\Fc(\pi_{1,p})}\le\GL_n(\overline\Q_\ell)\]
if $A=\overline\Q_\ell$, where $\overline{\cdot}$ denotes Zariski closure, for $\pi_{1,p}^\geom:=\Gal(\F_p(T)^\sep/\overline\F_p(T))$.\\

The works of Katz (see e.g. \cite{KatzGKM,KatzESDE,KatzSarnak91}) contain the determination of the monodromy groups over $\overline\Q_\ell$ of many sheaves of interest, such as Kloosterman sheaves. An important input is the fact that, for pointwise pure of weight $0$ sheaves, the connected component of the geometric monodromy group is a semisimple algebraic group by a result of Deligne.

The determination of discrete monodromy groups is usually more difficult, since they have far less structure.

\subsection{Coherent families}

  \begin{definition}\label{def:coherentFamily}
    Let $\Lambda$ be a set of valuations on $\Oc$ and let $U\subset\P^1/\F_p$ be an open affine subset. A family $(\Fc_\lambda)_{\lambda\in\Lambda}$, where $\Fc_\lambda$ is a sheaf of $\Oc_\lambda$-modules over $\F_p$ with maximal open of lissité $U$, is \textit{coherent} if:
  \begin{enumerate}
  \item\label{item:def:coherentFamily1} It forms a \textit{compatible system}: if $\rho_\lambda: \pi_{1,p}\to\GL_n(\Oc_\lambda)$ is the representation corresponding to $\Fc_\lambda$, then for every $\lambda\in\Lambda$, every finite extension $\F_{q}/\F_p$ and every $x\in U(\F_{q})$, the characteristic polynomial
    \[\charpol\rho_\lambda(\Frob_{x,q})\in\Oc_\lambda[T]\]
    lies in $E[T]$ and does not depend on $\lambda$.
  \item\label{item:coherentFamily2} There exists $G\in\{\SL_m,\Sp_{2m}\}$ such that for every $\lambda\in\Lambda$ corresponding to a prime ideal $\lf\normal\Oc$, the arithmetic and geometric monodromy groups of $\widetilde\Fc_\lambda:=\Fc_\lambda\pmod{\lf}$ coincide and are conjugate to $G(\F_\lambda)$. We call $G$ the \emph{monodromy group structure} of the family.    
  \end{enumerate}
  The \emph{conductor} of the family is defined to be $\sup_{\lambda\in\Lambda} \cond(\widetilde\Fc_\lambda)$, where 
    \[\cond(\widetilde\Fc_\lambda)=n+|\Sing(\widetilde\Fc_\lambda)|+\sum_{x\in\Sing(\widetilde\Fc_\lambda)} \Swan_x(\widetilde\Fc_\lambda) \hspace{0.2cm} (\lambda\in\Lambda)\]
    is the conductor defined by Fouvry--Kowalski--Michel (see e.g. \cite{AlgebraicTwists}).
  \end{definition}

  \begin{remark}
    Here, the prime $p$ is fixed, and the bounds of type \eqref{eq:density} would concern the trace functions on $\F_q$ obtained for every power $q$ of $p$. However, it may also make sense to vary $p$ (e.g. for Kloosterman sums of fixed rank, exponential sums \eqref{eq:expsumGen} coming from the reduction of integer polynomials, etc.), and the conductor will allow to control this dependency. See also Remark \ref{rem:exampleKl}.
  \end{remark}

  If $(\Fc_\lambda)_{\lambda\in\Lambda}$ is a compatible system as above, then in particular the trace function $t=t_{\Fc_\lambda}: \F_q\to\Oc_\lambda$ (as the opposite of the coefficient of order $n-1$ in the characteristic polynomial) is independent from $\lambda$ and takes values in $E$. More precisely,
\begin{equation}
  \label{eq:codomaintLocalization}
  t(\F_q)\subset\bigcap_{\lambda\in\Lambda} \Oc_\lambda\cap E=\bigcap_{\lf\in\Lambda} \Oc_\lf=\left(\Spec(\Oc)-\Lambda\right)^{-1}\Oc\subset E,
\end{equation}
where $\Oc_\lf$ is the localization at the ideal $\lf$ corresponding to the valuation $\lambda$.

%To simplify the notations, we will often drop the index $\lambda$ in the family and leave it implicit.

\subsubsection{Fourier transforms and coherent families}

The sheaves we will consider arise by $\ell$-adic Fourier transforms, as developed by Deligne, Laumon and others (see \cite[Section 7.3]{KatzESDE}, \cite[Chapter 5]{KatzGKM}), corresponding to the discrete Fourier transform on the level of trace functions.

This often results in sheaves with large classical monodromy groups, which is part of Condition \ref{item:coherentFamily2} above.

Concerning Condition \ref{item:def:coherentFamily1} and the conductor, we recall:
\begin{lemma}\label{lemma:compSysFT}
  Let us assume that $\Q(\zeta_{4p})\le E$ and let $\psi:\F_p\to\C$ be a nontrivial additive character. If $(\Fc_\lambda)_{\lambda\in\Lambda}$ is a compatible system of Fourier sheaves\footnote{See \cite[7.3.5]{KatzESDE} for the relevant definitions.} of $\Oc_\lambda$-modules over $\F_p$, then the family $(\FT_\psi(\Fc_\lambda))_{\lambda\in\Lambda}$ is compatible as well and $\cond(\widetilde{\FT_\psi}(\Fc_\lambda))\ll \cond(\widetilde\Fc_\lambda)^2$, where $\FT_\psi$ denotes the normalized Fourier transform with respect to $\psi$.
\end{lemma}
\begin{proof}
    Let $\Fc=\Fc_\lambda$ and $\Gc=\FT_\psi(\Fc)$. By construction, for every finite extension $\F_q/\F_p$ and every $a\in U_\Gc(\F_q)$, the reverse characteristic polynomial $\det \left(1-\Frob_{a,q}T\mid \Gc_{\overline\eta}\right)$ is equal to
    \[\prod_{i=0}^{2} \det \left(1-\Frob_qT\mid H^i_c(U_\Gc\times\overline\F_p,\Fc\otimes\Lc_{\psi(ax)})\right)^{(-1)^{i+1}}_,\]
    where $\Lc_{\psi(ax)}$ denotes an Artin--Schreier sheaf and $H_c^i$ the $i$th $\ell$-adic cohomology group with compact support. By the Grothendieck--Lefschetz trace formula \cite[Exposé 2]{DelEC}, this is $\exp \left(\sum_{n\ge 1} S(a,n)T^n/n\right)$, where $S(a,n)=\sum_{x\in U_\Gc(\F_{q^n})} t_{\Fc,q^n}(x)\psi(\tr(ax))$ has image in $E$ and does not depend on $\lambda$ by hypothesis, whence the conclusion.

    The assertion on the conductors can be found in \cite[Proposition 8.2]{AlgebraicTwists}, along with \cite[Remark 1.10]{KatzGKM}.
\end{proof}

\subsubsection{Examples}

For the examples below, we let $E=\Q(\zeta_{4p})$, with ring of integers $\Oc=\Z[\zeta_{4p}]$.

\begin{proposition}[Kloosterman sheaves]\label{prop:Klc}
  Let $n\ge 2$ be a fixed integer coprime to $p$. For
  \[\Lambda_n=\{\lambda \ \ell\text{-adic valuation on } \Oc : \ p\neq\ell\gg_n 1, \ \ell\equiv 1\pmod{4}, \ \F_\lambda=\F_\ell\},\]
  there exists a coherent family $(\Klc_{n,\lambda})_{\lambda\in\Lambda_n}$ of sheaves of $\Oc_\lambda$-modules over $\F_p$, with monodromy group structure
  \[\begin{cases}
      \SL_n&:n\text{ odd}\\
      \Sp_n&:n\text{ even},
    \end{cases}
  \]
conductor bounded by $n+3$, and such that the trace function $t_{\Klc_{n,\lambda},q}$ is equal to the Kloosterman sum $\Kl_{n,q}$ on $\F_q^\times$.
\end{proposition}
\begin{proof}
  The construction of the Kloosterman sheaves is due to Deligne (see \cite{KatzGKM} for the construction via recursive Fourier transforms). As already mentioned, the assertion on the integral monodromy groups over $\F_\lambda$ can be found in \cite{PGIntMonKS16}. They form a compatible system for $n$ fixed by Lemma \ref{lemma:compSysFT} applied recursively.
\end{proof}
\begin{remark}
  As an illustration of \eqref{eq:codomaintLocalization}, note that $\Kl_{n,q}: \F_q\to\Z[\zeta_{4p}]_{q^{(n-1)/2}}$.
\end{remark}

The following example, when unnormalized (hence replacing $\Oc_\lambda$ by $\Z_\ell$), was treated in \cite{KowLS06} and \cite{KowLargeSieve08}:

\begin{proposition}[Point counting on families of hyperelliptic curves]\label{prop:pointCounting}
  Let $f\in\Z[X]$ be a squarefree polynomial of degree $2g\ge 2$, and let $\Lambda$ be the set of $\ell$-adic valuations of $\Oc$ with $\ell\ge 3$. For $p$ large enough, there exists a coherent family $(\Fc_{f,\lambda})_{\lambda\in\Lambda}$ of $\ell$-adic sheaves of $\Oc_\lambda$-modules over $\F_p$, with monodromy group structure $\Sp_{2g}$, conductor depending only on $f$, and such that $t_{\Fc_{f,\lambda},q}(z)$ is given by \eqref{eq:familiesHE} when $f(z)\neq 0$.
\end{proposition}
\begin{proof}
  For the construction, see \cite[Section 10.1]{KatzSarnak91}, and normalize by a Tate twist. Because of this normalization, \cite[Theorem 10.1.16]{KatzSarnak91} and \cite[Lemma 10.1.9]{KatzSarnak91} show that the arithmetic and geometric monodromy group preserve the same symplectic pairing. Finally, \cite[Theorem 1.2]{Hall08} shows that the geometric monodromy group is $\Sp_{2g}$.
\end{proof}

\subsection{The large sieve for Frobenius}
\begin{theorem}\label{thm:largeSieve}
  Let $\Lambda$ be a set of valuations (or equivalently prime ideals) on $\Oc$. Given $L\ge 1$, we write $\Lambda_L$ for the set of valuations in $\Lambda$ corresponding to ideals of norm at most $L$. Let $(\Fc_\lambda)_{\lambda\in\Lambda}$ be a coherent family, with monodromy group structure $G$, where $\widetilde\Fc_\lambda$ corresponds to a representation
   \[\rho_\lambda: \pi_{1,p}\to\GL_n(\Oc_\lambda)\to\GL_n(\F_\lambda).\]
  For every $\lambda\in\Lambda$, let $\Omega_\lambda\subset G(\F_\lambda)$ be a conjugacy-invariant subset. Then, for all $L\ge 1$,
  \begin{equation*}
    \label{eq:largeSieve0}
    \frac{|\{x\in U_{\Fc_\lambda}(\F_q) : \rho_\lambda(\Frob_{x,q})\not\in\Omega_\lambda\textnormal{ for all }\lambda\in\Lambda_L\}|}{q}\ll \left(1+\frac{L^B}{q^{1/2}}\right)\frac{1}{P(L)},
  \end{equation*}
  where the implied constant depends only on the conductor of the family, and
  \begin{equation}
    \label{eq:BGeneral}
    P(L)=\sum_{\lambda\in\Lambda_L} \frac{|\Omega_\lambda|}{|G(\F_\lambda)|}, \ B=
  \begin{cases}
    \frac{2n^2+n-1}{2}&: G=\SL_n\\
    \frac{2n^2+3n+4}{4}&: G=\Sp_n\ (n\text{ even}).
    %\frac{2n^2-n+3}{4}&: G=\SO_n \ (n\text{ odd}).
  \end{cases}
  \end{equation}
\end{theorem}
\begin{proof}
      This is a variant of \cite[Proposition 3.3]{KowLS06} (see also \cite[Chapter 8]{KowLargeSieve08}). For $\lambda,\lambda'\in\Lambda$ distinct, the product map $\pi_{1,p}\to G(\F_\lambda)\times G(\F_{\lambda'})$ is surjective by \cite[Corollary 2.6]{KowLS06} (a variant of Goursat's Lemma), which extends with no modification to the case where $\F_\lambda$ and $\F_{\lambda'}$ do not necessarily have prime order (see \cite[Part III]{TesMal11}). By \cite[Corollary 24.6]{TesMal11}, $B=1+\dim(G)+\rank(G)/2$.
\end{proof}

\begin{remark}
  Note that in the case $E=\Q(\zeta_{d})$ of the examples of Section \ref{subsec:expSumsCycl}, the size of the residue field $\F_\lambda$ corresponding to a prime ideal $\lf\normal\Z[\zeta_d]$ depends on the multiplicative order modulo $d$ of the prime $\ell$ above which $\lf$ lies (see \cite[Theorem 2.13]{Was97}). In particular, if $d=4p$, then $|\F_\lambda|$ depends on $p$. This is a new phenomenon compared to the degree $1$ case (i.e. $\Oc_\lambda=\Z_\ell$) studied in \cite{KowLS06} and \cite{KowLargeSieve08}.
\end{remark}
\begin{remark}
  The case of orthogonal monodromy group structures (that would appear in some variants of the examples in Section \ref{sec:examples}) is excluded in the definition of a coherent family, because the argument above does not apply in general: see the remark after \cite[Corollary 2.6]{KowLS06}. A similar difficulty arises in Theorem \ref{thm:monodromyLP} later on: see Remark \ref{rem:LP}\ref{rem:LPSO}.
\end{remark}

\section{Traces of random matrices and Gaussian sums}\label{sec:tracesrandommatrices}

In the next section, we will apply Theorem \ref{thm:largeSieve} to $\Omega_\lambda=\{g\in G(\F_\lambda): \tr(g)\not\in A_\lambda\}$, for some $A_\lambda\subset\F_\lambda$. In this section, we get estimates on the densities
\[P\left(\tr(g)\not\in A_\lambda\right):=\frac{|\Omega_\lambda|}{|G(\F_\lambda)|}.\]
% Thus
% \[P(L)=\sum_{\lambda\in\Lambda_L} P\left(\tr(g)\not\in A_\lambda\right),\]
% % where $P\left(\tr(g)\not\in A_\lambda\right):=|G(\F_\lambda)|^{-1}|\{g\in G(\F_\lambda) : \tr(g)\not\in A_\lambda\}|$.
% where the probability is for $g\in G(\F_\lambda)$ with respect to the counting measure.

By the orthogonality relations in $\F_\lambda$, we get the following:
\begin{proposition}\label{prop:probTrG}
  Let $G\le\GL_n(\F_\lambda)$ be a subgroup and $A\subset\F_\lambda$. Then
  \begin{eqnarray*}
    P\left(\tr(g)\in A\right)&:=&\frac{1}{|G|}\sum_{g\in G} 1_A(\tr(g))\\
                             &=&\frac{|A|}{|\F_\lambda|}+O\left(\max_{1\neq\psi\in\hat\F_\lambda} \left|\frac{1}{|G|}\sum_{g\in G} \psi(\tr(g))\right|\left|\sum_{x\in A} \psi(-x)\right|\right).
  \end{eqnarray*}
\end{proposition}
We expect, for nontrivial $\psi\in\hat\F_\lambda$,
\begin{equation}
  \label{eq:gaussianSumClassical}
  \frac{1}{|G|}\sum_{g\in G} \psi(\tr(g))\ll |\F_\lambda|^{-\alpha(G)}
\end{equation}
for some $\alpha(G)>0$, and similarly, if $A$ is ``well-distributed'' in $\F_\lambda$, we expect
\begin{equation}
  \label{eq:boundGaussSumA}
  \frac{1}{|A|}\sum_{x\in A} \psi(x)\ll|\F_\lambda|^{-\alpha(A)}
\end{equation}
for some $\alpha(A)>0$. In both cases, the bounds should be uniform with respect to all nontrivial $\psi\in\hat\F_\lambda$.

Under \eqref{eq:gaussianSumClassical} and \eqref{eq:boundGaussSumA}, Proposition \ref{prop:probTrG} becomes
\begin{equation}
  \label{eq:probTrG}
  P\left(\tr(g)\in A\right)=\frac{|A|}{|\F_\lambda|}\left(1+O \left(|\F_\lambda|^{-\alpha(G)-\alpha(A)+1}\right)\right).
\end{equation}

\subsection{Gaussian sums in linear groups \eqref{eq:gaussianSumClassical}}

\subsubsection{General result}

We start by a result that applies more generally to algebraic varieties in $\GL_n$.

\begin{proposition}\label{prop:boundGaussianSumClassical}
  Let $V=\bs V(\F_\lambda)$ for $\bs V\subset\GL_n$ an algebraic variety over $\F_\lambda$. The bound \eqref{eq:gaussianSumClassical} holds with $\alpha(V)=1/2$, uniformly for all nontrivial $\psi\in\hat\F_\lambda$, unless $\tr: V\to\F_\lambda$ is constant.
\end{proposition}
\begin{proof}
  Let $\ell'\neq\car(\F_\lambda)$ be an auxiliary prime and let us consider the restriction $\Lc$ of the Lang torsor $\Lc_{\psi\circ\tr}$ on $\A^{n^2}/\F_\lambda$ to $\bs V$ (see \cite[Example 7.17]{KiRu14}), as sheaf of $\Q_{\ell'}$-modules. By the Grothendieck--Lefschetz trace formula,
  \[\sum_{g\in V} \psi(\tr(g))=\sum_{i=0}^{2\dim \bs V} (-1)^i\tr\left(\Frob_{\F_\lambda}\mid H^i_c(\bs V\times\overline\F_\lambda,\Lc)\right).\]
  By Deligne's generalization of the Riemann hypothesis over finite fields \cite{Del2},
  \[\tr\left(\Frob_{\F_\lambda}\mid H^i_c(\bs V\times\overline\F_\lambda,\Lc)\right)\le |\F_\lambda|^{i/2}\dim H^i_c(\bs V\times\overline\F_\lambda,\Lc)\]
  for $0\le i\le 2\dim\bs V$, and by the coinvariant formula,
  \[\tr\left(\Frob_{\F_\lambda}\mid H^{2\dim\bs V}_c(\bs V\times\overline\F_\lambda,\Lc)\right)=0\]
  unless $\Lc$ is geometrically trivial, in which case $\tr: V\to\F_\lambda$ would be constant. Hence
  \[\left|\sum_{g\in V} \psi(\tr(g))\right|\le |\F_\lambda|^{\dim\bs V-1/2} \sum_{i=0}^{2\dim \bs V-1}\dim H^i_c(\bs V\times\overline\F_\lambda,\Lc).\]
  By \cite[Theorem 12]{KatzBetti01}, we find that
  \[\left|\sum_{g\in V} \psi(\tr(g))\right|\le 3|\F_\lambda|^{\dim\bs V-1/2}(2+d)^{n^2+r}\]
  if $\bs V$ is defined by $r$ polynomials of degree at most $d$. The conclusion follows by \cite[Corollary 24.6]{TesMal11}.
\end{proof}
\subsubsection{Classical finite groups of Lie type}

Using the Bruhat decomposition, D.S. Kim actually explicitly evaluated the Gaussian sums \eqref{eq:gaussianSumClassical} for classical finite groups of Lie type (see e.g. \cite{Kim97,Kim98}). The expressions involve hyper-Kloosterman sums, and applying Deligne's bound yields the following, which greatly improves Proposition \ref{prop:boundGaussianSumClassical}, in particular as $n$ grows:
\begin{proposition}\label{prop:boundGaussianSumClassicalKim}
  For $n\ge 1$ and $G=\GL_n(\F_\lambda)$, $\SL_n(\F_\lambda)$, $\Sp_{2n}(\F_\lambda)$, $\SO_{2n}^\pm(\F_\lambda)$ and $\SO_{2n+1}(\F_\lambda)$, the bound \eqref{eq:gaussianSumClassical} holds with $\alpha(G)\ge 1$ given in Table \ref{table:boundsGaussianSumsClassical}.
\end{proposition}
\begin{proof}
  See \cite[Proposition 6.28]{PGDistrTFCycl16}.
\end{proof}
  \begin{table}
    \centering
    \begin{tabular}{c|c}
      $G$&$\alpha(G)$\\\hline
      $\GL_n$&$\frac{n(n-1)}{2}$\\
      $\SL_n$&$\frac{n^2-1}{2}$\\
      $\Sp_n$, $\SO_n^-$ \textup{(}$n$\textup{ even)}&$\frac{n(n+2)}{8}$\\
      $\SO_n$ \textup{(}$n$\textup{ odd)}&$\frac{n^2-1}{8}$\\
      $\SO_n^+$ \textup{(}$n$\textup{ even)}&$\frac{n(n-2)}{8}$
    \end{tabular}
    \caption{Cancellation for Gaussian sums over finite groups of Lie type.}
    \label{table:boundsGaussianSumsClassical}
  \end{table}

\subsection{Gaussian sums in $\F_\lambda$}\label{subsec:gaussianSumsFlambda}

Let us now consider Bound \eqref{eq:boundGaussSumA} for various subsets $A\subset\F_\lambda$.

\subsubsection{Squares} Let $A=\F_\ell^{\times 2}$ be the subgroup of squares in $\F_\ell^\times$ with $\ell>2$. Using the Legendre symbol and the evaluation of quadratic Gauss sums, we get that \eqref{eq:boundGaussSumA} holds with $\alpha(A)=1/2$, uniformly for all nontrivial $\psi\in\hat\F_\ell$, corresponding to square-root cancellation since $|A|=(\ell-1)/2$.

\subsubsection{Powers/Multiplicative subgroups}
More generally, we have:
\begin{proposition}
  For $\alpha\in(0,1/2)$, Bound \eqref{eq:boundGaussSumA} holds for any subgroup $H\le\F_\lambda^\times$ such that $|H|\ge|\F_\lambda|^{1/2+\alpha}$, uniformly for all nontrivial $\psi\in\hat\F_\lambda$.
\end{proposition}
\begin{proof}
  This follows for example from the bound $\sum_{x\in H} \psi(x)\ll |\F_\lambda|^{1/2}$ that is deduced from Deligne's extension of the Riemann hypothesis over finite fields (see \cite[Proposition 5.7]{PGDistrTFCycl16}).
\end{proof}

\begin{example}
  For $m\ge 2$ fixed and $H=\F_\lambda^{\times m}$ the subgroup of $m$th powers, the condition $|H|\ge|\F_\lambda|^{1/2+\alpha}$ holds as soon as $|\F_\lambda|$ is large enough, since $|H|=\frac{|\F_\lambda|-1}{(m, |\F_\lambda|-1)}$.
\end{example}

\begin{remark}
  When $|H|$ is arbitrarily small (say $|H|\ge|\F_\lambda|^{\delta}$ for some $\delta>0$), the works of Bourgain and others (see e.g. \cite{BouCh06}) give \eqref{eq:boundGaussSumA} for some $\alpha=\alpha(\delta)>0$, up to some necessary restrictions if $\delta\le 1/2$ and $\F_\lambda\neq\F_\ell$.
\end{remark}

\subsubsection{Definable subsets}\label{subsub:definableSubsets}
For $R$ a ring and $\varphi(x)$ a first-order formula in one variable in the language of rings, we define $\varphi(R)=\{a\in R : \varphi(a)\text{ holds}\}$.

\begin{example}
  For $\varphi(x)=(\exists y : x=y^2)$, the set $\varphi(R)$ is the subset of squares, as in the previous section. More generally, we can take $\varphi(x)=(\exists y: x=f(y))$ for any polynomial $f\in \Z[Y]$.
\end{example}

We recall:
\begin{theorem}[Chatzidakis--van den Dries--Macintyre \cite{CDM92}]\label{thm:CDM}
  For every formula $\varphi(x)$ in one variable in the language of rings, there exists a finite set $C(\varphi)\subset(0,1]\cap\Q$ such that for every finite field $\F_\lambda$,
  \begin{eqnarray}
    |\varphi(\F_\lambda)|&=&C(\lambda,\varphi)|\F_\lambda|+O_\varphi(|\F_\lambda|^{1/2})\label{eq:CDM1}\\
    &&\text{with }C(\lambda,\varphi)\in C(\varphi)\text{, or}\nonumber\\
    |\varphi(\F_\lambda)|&\ll_\varphi& |\F_\lambda|^{-1/2}.\label{eq:CDM2}
  \end{eqnarray}
  The implied constants depend only on $\varphi$.
\end{theorem}
  
  The following combined with Theorem \ref{thm:CDM} shows that Gaussian sums over definable subsets exhibit square-root cancellation:

\begin{theorem}[{\cite[Theorem 1, Corollary 12, Remark 19]{Kow07}}]\label{prop:expsumsdef}
  Let $\varphi(x)$ be a formula in one variable in the language of rings such that $|\varphi(\F_\lambda)|$ is not bounded as $|\F_\lambda|\to+\infty$. Then, if $\psi\in\hat\F_\lambda$ is nontrivial, the bound \eqref{eq:boundGaussSumA} for $A=\varphi(\F_\lambda)$ holds with $\alpha(A)=1/2$, with an implied constant depending only on $\varphi$.
  \end{theorem}

  \subsubsection{Images of polynomials}\label{subsubsec:imagesPolynomials}

  When $\varphi(x)=(\exists y : x=f(y))$ for some polynomial $f\in\Z[X]$, Theorem \ref{thm:CDM} also appears in \cite{BSD59} (using the Weil conjectures for curves).

  \begin{proposition}[{\cite[Theorem 1, Lemma 1]{BSD59}}]\label{prop:BSD}
    Let $f\in\Z[X]$ be of degree $d\ge 2$ and such that the Galois group of $f(X)-y\in\C(y)[X]$ over $\C(y)$ is equal to $\Sf_d$. Then \eqref{eq:CDM1} for $\varphi(x)=(\exists y : x=f(y))$ and a finite field $\F_\lambda$ of characteristic $\ell\gg_f 1$ holds with
    \begin{equation*}
     \label{eq:Cvarphif}
      C(\lambda,\varphi)=\sum_{n=1}^{\deg(f)} \frac{(-1)^{n+1}}{n!}\in(0,1).
    \end{equation*}
  \end{proposition}

  This is extended to $f\in\F_\lambda(X)$ in \cite{Cohen70}.
  
  \begin{remarks}\label{rem:almostAll}
    \begin{enumerate}
    \item See \cite[p. 422]{BSD59} for sufficient conditions to verify the hypothesis of Proposition \ref{prop:BSD}.
    \item By \cite{vdW34} or \cite{Gall73}, the hypothesis of Proposition \ref{prop:BSD} holds for almost all monic $f\in\Z[X]$ of degree $d\ge 2$, with respect to the terminology of Footnote \ref{fn:almostallpol}, p. \pageref{fn:almostallpol}.
    \end{enumerate}
  \end{remarks}
  
\section{Zero-density estimates for trace functions in algebraic subsets}\label{sec:zerodensity}

We continue to fix a number field $E$ with ring of integers $\Oc$.

\subsection{General result}

\begin{proposition}\label{prop:largeSieve}
  Let $\Lambda$ be a set of valuations on $\Oc$ and let $t: \F_q\to E$ be the trace function over $\F_q$ associated to a coherent family $(\Fc_\lambda)_{\lambda\in\Lambda}$ of sheaves of $\Oc_\lambda$-modules over $\F_p$, with monodromy group structure $G$. For $A\subset E$ and $\lambda\in\Lambda$ corresponding to a prime ideal $\lf$ of $\Oc$, we denote by $A_\lambda\subset\F_\lambda$ the reduction of $A\cap\Oc_\lf$ modulo $\lf$. Assume that
  \begin{equation}
    \label{eq:localDensities}
    \sup_{\lambda\in\Lambda} \frac{|A_\lambda|}{|\F_\lambda|}<1.
  \end{equation}
  Then
  \begin{equation}
    \label{eq:largeSieve}
    P\big(t(x)\in A\big)\ll \frac{1}{|\Lambda_L|}\text{ with }L=\floor{q^{\frac{1}{2B}}},
  \end{equation}
  where $B>0$ is as in Theorem \ref{thm:largeSieve}, with an implied constant depending only on the conductor of the family and on the left-hand side of \eqref{eq:localDensities}.
\end{proposition}
\begin{proof}
  For every $\lambda\in\Lambda$, we may reduce $t: \F_q\to\Oc_\lf\le\Oc_\lambda$ to $\tilde t: \F_q\to\F_\lambda$, so that
  \[P(t(x)\in A)\le\frac{|\{x\in\F_q : \tilde t(x)\in A_\lambda\text{ for all }\lambda\in\Lambda_L\}|}{q}.\]
  By Theorem \ref{thm:largeSieve} with
  \[\Omega_\lambda=\{g\in G(\F_\lambda) : \tr{g}\not\in A_\lambda\} \hspace{0.3cm} (\lambda\in\Lambda),\]
  which are clearly conjugacy-invariant, we get
  \[P(t(x)\in A)\ll\left(1+\frac{L^B}{q^{1/2}}\right)\frac{1}{P(L)},\]                
  where $P(L)=\sum_{\lambda\in\Lambda_L} P(\tr(g)\not\in A_\lambda)$. By \eqref{eq:probTrG} (Proposition \ref{prop:probTrG}),
  \[P(\tr(g)\in A_\lambda)=\frac{|A_\lambda|}{|\F_\lambda|}\left(1+O \left(\frac{1}{|\F_\lambda|^{\alpha(G)+\alpha(A_\lambda)-1}}\right)\right)\ll\frac{|A_\lambda|}{|\F_\lambda|},\]
  since $\alpha(G)\ge 1$ by Proposition \ref{prop:boundGaussianSumClassicalKim}. Therefore, we get that for any $L\ge 1$,
  \[P\big(t(x)\in A\big)\ll \left(1+\frac{L^B}{q^{1/2}}\right)|\Lambda_L|^{-1}\left(1-\max_{\lambda\in\Lambda_L} \frac{|A_\lambda|}{|\F_\lambda|}\right)^{-1}.\]
\end{proof}

\begin{remark}
  If we assume more generally that the monodromy group of $\widetilde\Fc_\lambda$ is $\bs G(\F_\lambda)$ for $\bs G\le\GL_n$ any linear group over $\F_\lambda$, the results hold if $\alpha(A_\lambda)\ge 1/2$ for all $\lambda\in\Lambda$, by Proposition \ref{prop:boundGaussianSumClassical}. Interestingly, in the case of $\SL_n$, $\Sp_{2n}$ and $\SO_n^\pm$, Proposition \ref{prop:boundGaussianSumClassicalKim} gives much more cancellation, so that we do not need information about the $\alpha(A_\lambda)$.
\end{remark}

To apply Proposition \ref{prop:largeSieve}, we need the local densities assumption \eqref{eq:localDensities} and lower bounds on $|\Lambda_L|$. We treat these aspects in the following subsections.

\subsection{Lower bounds on $|\Lambda_L|$}\label{subsec:lowerBoundsLambdaL}

For our applications, we will mainly consider $\Lambda$ to be either:
\begin{examples}\label{ex:Lambda}
  \begin{enumerate}
  \item\label{item:exLambdaI} The full set $\Lambda_{0,p}$ of valuations on $\Oc$ not lying above the $p$-adic valuation.
  \item\label{item:exLambdaII} For $m\ge 2$ and $C\subset(\Z/m)^\times$, the set of valuations $\lambda\in\Lambda_{0,p}$ such that $|\F_\lambda|\in C$.
  \item The restriction of these to ideals having degree $1$ over $\Q$.
  \end{enumerate}
\end{examples}

More generally, let $F/E$ be a fixed finite Galois extension of number fields with Galois group $H$, $C\subset H$ be a conjugacy-stable subset, and
  \begin{eqnarray}
    \label{eq:chebotarev}
    \Lambda(C)&=&\{\lf\normal E\text{ prime, not ramified in } F : \Frob_\lf\in C\}\\
    \Lambda_1(C)&=&\{\lf\in\Lambda(C) \text{ of degree 1 over }\Q\}\nonumber.
    \nonumber
  \end{eqnarray}
  Example \ref{ex:Lambda} \ref{item:exLambdaI} then corresponds to $E=F$, while \ref{item:exLambdaII} corresponds to $F=E(\zeta_m)$ with $H\cong(\Z/m)^\times$.

By Chebotarev's density theorem, if $E$ and $F$ are fixed,
\begin{equation}
  \label{eq:ChebLowerBound}
  |\Lambda(C)_L|\ge |\Lambda_1(C)_L|\gg \frac{|C|}{|H|}\frac{L}{\log{L}} \hspace{0.5cm} (L\to+\infty)
\end{equation}
with an absolute implied constant. Hence, if $F$ and $E$ do not depend on $p$, \eqref{eq:largeSieve} is
\begin{equation}
  \label{eq:boundLargeSieveCheb}
  P\big(t(x)\in A\big)\ll_{C,H} \frac{\log{q}}{B q^{1/(2B)}}\to 0 \hspace{0.5cm} (q=p^e\to+\infty).
\end{equation}
If $E$ and/or $F$ depend on $p$ (e.g. for Kloosterman sums, where $E=\Q(\zeta_{4p})$), we must either fix $p$ or deal with uniformity with respect to $E$ and $F$. We discuss this situation in the following paragraphs.

\subsubsection{Uniformity in the prime ideal theorem} By \cite{Fried80} (extending Chebychev's method to number fields), if $E/\Q$ is normal\footnote{In Friedlander's paper, it is only assumed that $E$ is in a tower of normal extensions. If $E/\Q$ is itself normal, we can improve the result by using more a precise version of Stark's estimates \cite{Sta74} on the residue at $1$ of the Dedekind zeta function of $E$.}, then
\[\pi_E(L)=|\{\lf\normal E\text{ prime} : N(\lf)\le L\}|\gg_\varepsilon \frac{L}{\log(2L)^{1+\varepsilon}\Delta_E^{1/2+\varepsilon}}\]
for $\Delta_E=|\disc_\Q(E)|$, and any $\varepsilon>0$ if $n_E=[E:\Q]\gg_\varepsilon 1$. This is nontrivial only when $L\gg \Delta_E^{1/2+\varepsilon'}$ for some $\varepsilon'>0$.
\subsubsection{Uniformity in Chebotarev's density theorem}\label{subsubsec:uniformityCheb}
The unconditional results due to Lagarias--Odlyzko and Serre (see \cite[Section 2.2]{Ser81}) show that \eqref{eq:ChebLowerBound} holds with an absolute implied constant under the restriction $\log{L}\gg n_E(\log\Delta_E)^2$.

Assuming the generalized Riemann hypothesis (GRH) for the Dedekind zeta function of $E$, this range can be improved to $L\gg(\log\Delta_E)^{2+\varepsilon}$ for an arbitrary $\varepsilon>0$ (see \cite[Section 2.4]{Ser81}).

\subsubsection{Cyclotomic fields}

If $E=\Q(\zeta_d)$, $F=E(\zeta_m)$ are cyclotomic fields, it is possible to improve the unconditional uniform range in Chebotarev's density theorem by relying on estimates for primes in arithmetic progressions.
\begin{proposition}\label{prop:chebotarevCyclotomic}
    For $d,m\ge 1$ coprime integers, let $E=\Q(\zeta_d)$ and $F=E(\zeta_m)$. For $C\subset\Gal(F/E)\cong(\Z/m)^\times$, we have
    \[|\Lambda(C)_L|\ge |\Lambda_1(C)_L|\gg\frac{|C|L}{(dm)^\varepsilon\varphi(m)\log{L}}\]
    when either:
    \begin{enumerate}
    \item \label{item:prop:chebotarevCyclotomic1} $\varepsilon>0$ and $L\ge (dm)^8$, or
    \item \label{item:prop:chebotarevCyclotomic2} under GRH, $\varepsilon=0$ and $L\ge (dm)^{2+\varepsilon'}$ for some $\varepsilon'>0$.
    \end{enumerate}
\end{proposition}
\begin{proof}
  Since every unramified rational prime $\ell$ of inertia/residual degree $f_\ell$ (equal to the order of $\ell$ in $(\Z/d)^\times$) gives rise to $\varphi(d)/f_\ell$ prime ideals with norm $\ell^{f_\ell}$,
  \begin{eqnarray*}
    |\Lambda(C)_L|&=&\varphi(d)\sum_{f\mid \varphi(d)}\frac{|\{\ell\le L^{1/f}\text{ prime } : \ell\nmid\Delta_E, \, f_\ell=f, \ \ell^f\in C\}|}{f}.
  \end{eqnarray*}
  The summand with $f=1$ gives:
  \begin{eqnarray*}
    |\Lambda_1(C)_L|&\ge& \varphi(d)|\{\ell\le L\text{ prime } : \ell\nmid\Delta_E, \ \ell\equiv 1\pmod{d}, \ \ell\in C\}|.
  \end{eqnarray*}
  If $(d,m)=1$, then by the Chinese remainder theorem
  \[|\Lambda_1(C)_L|\ge\varphi(d)\left[\sum_{c\in C}\pi(c,dm,L)-\omega(d)\right],\]
  where $\pi(a,d,L)=|\{\ell\le L\text{ prime } : \ell\equiv a\pmod{d}\}|$ for $a\in(\Z/d)^\times$. Uniformly, one has
\begin{equation}
  \label{eq:DirichletExpectUniform}
  \pi(a,d,L)\gg \frac{L}{\varphi(d)d^{\varepsilon}\log{L}}
\end{equation}
under \ref{item:prop:chebotarevCyclotomic1} (by \cite[Theorem 3.3]{May13}, using Linnik-type arguments) or \ref{item:prop:chebotarevCyclotomic2} assuming GRH.
\end{proof}

\begin{remark}\label{rem:primesDeg1Density}
  Similarly, this shows that for a Galois extension $E/\Q$, the set of prime ideals with inertia degree $1$ has natural density $1$, so we cannot hope to substantially improve the lower bound by taking into account the $f>1$ in the proof of Proposition \ref{prop:chebotarevCyclotomic}.
\end{remark}

\begin{remarks}\label{rem:montgomery}
  \begin{enumerate}
  \item By the Bombieri--Vinogradov theorem, the range \ref{item:prop:chebotarevCyclotomic2} in \eqref{eq:DirichletExpectUniform} holds unconditionally for all $a$ on average over $d$.
  \item By a conjecture of Montgomery, one may be able to take $\varepsilon=0$ and $L\gg(dm)^{1+\delta}$ for any $\delta>0$. By Barban--Davenport--Halberstam, Montgomery, and Hooley, this holds true in \eqref{eq:DirichletExpectUniform} on average over $d$ and $a$.
  \end{enumerate}
\end{remarks}

\subsection{Explicit zero-density estimates}
The results from the previous section along with Proposition \ref{prop:largeSieve} give:
\begin{proposition}\label{prop:largeSieveExplicit}
  Under the hypotheses of Proposition \ref{prop:largeSieve} and \eqref{eq:localDensities}, with $E/\Q$ normal, $F/E$ a finite Galois extension with Galois group $H$, a conjugacy-invariant subset $C\subset H$ and $\Lambda=\Lambda(C)$ or $\Lambda_1(C)$ as in \eqref{eq:chebotarev}, we have that for any $\varepsilon>0$:
  \begin{enumerate}
  \item If $F=E$ is normal,
    \[P\big(t(x)\in A\big)\ll_\varepsilon \frac{\Delta_E^{1/2+\varepsilon}(\log{q})^{1+\varepsilon}}{B^{1+\varepsilon}q^{1/(2B)}},\]
    which is nontrivial when $\Delta_E^{B+\varepsilon'}=o(q)$ for some $\varepsilon'>0$.
  \item Under GRH, if $q\ge(\log\Delta_E)^{2B+\varepsilon}$,
    \[P\big(t(x)\in A\big)\ll_\varepsilon\frac{m\log{q}}{|C|Bq^{1/(2B)}}\ll_{m,C} \frac{\log{q}}{B q^{1/(2B)}}.\]
  \item Assume that $E=\Q(\zeta_d)$ and $F=E(\zeta_m)$ with $(d,m)=1$. If $q\ge (dm)^{16B}$, then
    \begin{equation*}
      P\big(t(x)\in A\big)\ll_\varepsilon\frac{m(dm)^\varepsilon\log{q}}{|C|B q^{1/(2B)}}\ll_{m,C} \frac{d^\varepsilon\log{q}}{B q^{1/(2B)}}.
    \end{equation*}
  \end{enumerate}
  The implied constants depend only on the conductor of the family and the quantities indicated.
\end{proposition}

\subsubsection{The case $E=\Q(\zeta_{4p})$}
For exponential sums, we are interested in the case $E=\Q(\zeta_{4p})$, where $n_E=2(p-1)$ and $\Delta_E=4^{2p-3}p^{2(p-2)}$.

The restrictions $q\gg g(E)$ (for some $g(E)=g(n_E,\Delta_E)\ge 1$) of Proposition \ref{prop:largeSieveExplicit} impose limitations on the range of $e,p$ when $q=p^e\to+\infty$:
\begin{corollary}\label{cor:largeSieveExplicitQzeta4p}
  Under the hypotheses of Proposition \ref{prop:largeSieveExplicit} for $E=\Q(\zeta_{4p})$ and $F=E(\zeta_m)$ with $(m,4p)=1$, we have
  \[P\big(t(x)\in A\big)\ll_\varepsilon\frac{m(pm)^\varepsilon\log{q}}{|C|Bq^{1/(2B)}}\ll_{m,C} \frac{p^\varepsilon\log{q}}{B q^{1/(2B)}}\to 0 \hspace{0.5cm} (q=p^e\to+\infty)\]
  when either
  \begin{center}
      \begin{enumerate*}
      \item $\varepsilon>0$ and $e\ge 16B$, or
      \item under GRH, $\varepsilon=0$ and $e>4B$.
      \end{enumerate*}
    \end{center}
    The implied constants depend only on the conductor of the family and the quantities indicated.
\end{corollary}

\begin{remarks}\label{rem:methodExtende1}
  \begin{enumerate}
  \item Had we not taken advantage of the fact that $E$ is a cyclotomic field, the unconditional results mentioned in Section \ref{subsubsec:uniformityCheb} would have forced to take $q=p^e\to+\infty$ with $e\gg p$.
  \item \label{rem:item:methodExtende1} Under Montgomery's conjecture mentioned in Remarks \ref{rem:montgomery}, we may take $\varepsilon=0$ and $e>2B$. Without an improvement in the error term of the large sieve bound \eqref{eq:largeSieve}, $e=2B+1\ge 10$ is the minimal value the method could handle.
  \end{enumerate}
\end{remarks}

\subsection{Local densities}\label{subsec:localDensities}

In this section, we finally give examples of sets $A\subset E$ for which the local densities assumption \eqref{eq:localDensities} holds.

\subsubsection{Powers/finite index subgroups}

\begin{proposition}\label{prop:powersFiniteIndex}
  Let $E$, $\Oc$ be as in Proposition \ref{prop:largeSieve} and for $m\ge 2$, let
  \[\Lambda=
    \begin{cases}
      \{\lambda\in\Lambda_{0,p} : |\F_\lambda|\equiv 1\pmod{m}\}&: m\text{ odd}  \\
      \{\lambda\in\Lambda_{0,p} : \text{not lying above }2\}&: m\text{ even}.
    \end{cases}\]
  Then \eqref{eq:localDensities} holds for $A=E^m\subset E$.
\end{proposition}
\begin{proof}
  We have $A_\lambda=\F_\lambda^m$ and for $|\F_\lambda|\ge 3$,
  \begin{eqnarray*}
    \frac{|A_\lambda|}{|\F_\lambda|}&=&\left(1-\frac{1}{|\F_\lambda|}\right)\frac{1}{(|\F_\lambda^\times|,m)}+\frac{1}{|\F_\lambda|}\ll
                                        \begin{cases}
                                          \frac{1}{m}+\frac{1}{|\F_\lambda|}&: m\text{ odd}\\
                                          \frac{1}{2}+\frac{1}{|\F_\lambda|}&: m\text{ even}.
                                        \end{cases}
  \end{eqnarray*}
\end{proof}
Note that the set $\Lambda$ in Proposition \ref{prop:powersFiniteIndex} is of the form given in Example \ref{ex:Lambda} \ref{item:exLambdaII}.
\subsubsection{Definable subsets}
\begin{proposition}\label{prop:definableSubsets}
  Let $E$, $\Oc$ and $\Lambda$ be as in Proposition \ref{prop:largeSieve} and let $\varphi(x)$ be a first order formula in one variable in the language of rings such that:
  \begin{enumerate}
  \item\label{item:cor:definableSubsets1} Neither $|\varphi(\F_\lambda)|$ nor $|\neg\varphi(\F_\lambda)|$ are bounded as $|\F_\lambda|\to+\infty$, where $\neg$ denotes negation.
  \item\label{item:cor:definableSubsets2} For every $\lambda\in\Lambda$ corresponding to an ideal $\lf$, $\varphi(E)\cap\Oc_\lf\pmod{\lf}$ is contained in $\varphi(\F_\lambda)$.
  \end{enumerate}
  Then \eqref{eq:localDensities} holds with $A=\varphi(E)\subset E$.
\end{proposition}
\begin{proof}
  Condition \ref{item:cor:definableSubsets2} implies that $A_\lambda\subset\varphi(\F_\lambda)$ for all $\lambda\in\Lambda$. Under condition \ref{item:cor:definableSubsets1}, Theorem \ref{thm:CDM} shows that
  \begin{eqnarray*}
    |\varphi(|\F_\lambda|)|&=&C_{\lambda,\varphi}|\F_\lambda|(1+o(1))\\
    |\neg\varphi(|\F_\lambda|)|&=&C_{\lambda,\neg\varphi}|\F_\lambda|(1+o(1))\\
                           &=&(1-C_{\lambda,\varphi})|\F_\lambda|(1+o(1))
  \end{eqnarray*}
  with $C_{\lambda,\varphi},C_{\lambda,\neg\varphi}\in (0,1]$. Hence, $C_{\lambda,\varphi}\neq 0,1$ for $|\F_\lambda|$ large enough, and $\limsup_{|\F_\lambda|\to+\infty}\frac{|A_\lambda|}{|\F_\lambda|}\le\limsup_{|\F_\lambda|\to+\infty}\frac{|\varphi(\F_\lambda)|}{|\F_\lambda|} \le \max C(\varphi)<1$, recalling that $C(\varphi)$ is finite.
\end{proof}
\begin{remark}\label{rem:definableSubsets}~
  Condition \ref{item:cor:definableSubsets2} of Proposition \ref{prop:definableSubsets} holds if both
  \begin{multicols}{2}
  \begin{enumerate}[(a)]
  \item\label{item:rem:definableSubsetsa} $\varphi(E)\cap\Oc_\lf\subset\varphi(\Oc_\lf)$, \hspace{0.5cm}and
  \item\label{item:rem:definableSubsetsb} $\varphi(\Oc_\lf)\pmod{\lf}\subset\varphi(\F_\lambda)$
  \end{enumerate}
  \end{multicols}
  \noindent hold. Note that:
  \begin{itemize}
  \item Condition \ref{item:rem:definableSubsetsa} holds when $\car(\F_\lambda)\gg_f 1$ if $\varphi(x)=(\exists y : x=f(y))$ for some $f\in\Z[X]$. Indeed, for $x\in E$, we have $\lambda(f(x))=\min(0, \deg(f)\lambda(x))$ if no coefficient of $f$ is divisible by $\car(\F_\lambda)$.
  \item Condition \ref{item:rem:definableSubsetsb} holds if $\varphi$ contains no negations or implications. On the other hand, for $\varphi(x)=\neg (\exists y : x=y^2)$, the reduction of a nonsquare in $\Oc$ may be a square in $\F_\lambda$.
  \end{itemize}
\end{remark}

\begin{example}[Images of polynomials]\label{ex:definablePolynomials}
  Consider the case $\varphi(x)=(\exists y : x=f(y))$ for $f\in\Z[X]$ of Section \ref{subsubsec:imagesPolynomials}. Then Proposition \ref{prop:definableSubsets} applies for
  \begin{itemize}
  \item almost all monic $f$ of fixed degree $d\ge 2$ (with respect to the terminology of Footnote \ref{fn:almostallpol}, p. \pageref{fn:almostallpol}), and
  \item all $f$ satisfying the Galois group condition of Proposition \ref{prop:BSD},
  \end{itemize}
  up to restricting to a cofinite subset of $\Lambda$. Indeed:
  \begin{itemize}
  \item By Proposition \ref{prop:BSD} and Remarks \ref{rem:almostAll}, Condition \ref{item:cor:definableSubsets1} of Proposition \ref{prop:definableSubsets} holds for almost all monic $f$ of fixed degree.
  \item Condition \ref{item:cor:definableSubsets2} holds if $\car(\F_\lambda)\gg_f 1$ by Remark \ref{rem:definableSubsets}.
  \end{itemize}
\end{example}
\section{Examples}\label{sec:examples}

\subsection{Kloosterman sums}

Proposition \ref{prop:LSKlPower}, given in the introduction, now follows directly from Corollary \ref{cor:largeSieveExplicitQzeta4p} with Proposition \ref{prop:Klc} and the local densities estimates from Proposition \ref{prop:powersFiniteIndex}.\\

Similarly, replacing the latter with Proposition \ref{prop:definableSubsets}, we obtain:
\begin{proposition}\label{prop:definableSubsetsKl}
  Let $\varphi(x)$ be a first-order formula in the language of rings as in Proposition \ref{prop:definableSubsets}. Then, for $n\ge 2$ and $\varepsilon>0$,
  \begin{equation}
    \label{eq:definableSubsetsKl}
    P\Big(\Kl_{n,q}(x)\in\varphi(\Q(\zeta_{4p}))\Big)\ll_{n,\varphi,\varepsilon} \frac{p^\varepsilon\log{q}}{B_nq^{1/(2B_n)}}\to 0
  \end{equation}
  when $q=p^e\to+\infty$ coprime to $n$ with $e\ge 16B_n$, for $B_n$ as in \eqref{eq:Bn}. The implied constant depends only on $n$, $\varphi$ and $\varepsilon$.
\end{proposition}
Proposition \ref{prop:LSKlf} is a particular case of the latter, using Example \ref{ex:definablePolynomials}.

\subsubsection{Results for unnormalized sums}\label{subsubsec:unnormalized} Replacing $A$ by $q^{(n-1)/2}A$ in Proposition \ref{prop:largeSieve} and using uniformity shows that the above results also hold for unnormalized Kloosterman sums.

\subsubsection{Galois actions}\label{subsubsec:GaloisKl}
When considering densities of the form \eqref{eq:definableSubsetsKl}, it is interesting to take into account the following Galois actions:

\begin{enumerate}
\item For all $\sigma\in\Gal(\F_q/\F_p)\cong\Z/e$ and $x\in\F_q^\times$,
  \[\Kl_{n,q}(x)=\Kl_{n,q}(\sigma(x)).\]
  The orbit of $x$ has size $\deg(x)\in\{1,\dots, e\}$. Fisher \cite[Corollary 4.25]{Fisher92} has actually shown that if $p>(2n^{2e}+1)^2$, the Kloosterman sums are distinct up to this action.
\item For $\sigma\in\Gal(\Q(\zeta_p)/\Q)\cong\F_p^\times$ corresponding to $c\in\F_p^\times$ and $x\in\F_q^\times$, we have
  \[\sigma(\Kl_{n,q}(x))=\Kl_{n,q}(c^nx).\]
  Moreover, orbits have size $\frac{p-1}{(p-1,n)}\in\{(p-1)/n,\dots,p-1\}$.
\end{enumerate}

If $\varphi$ is a first-order formula in the language of rings, let $A_p=\varphi(\Q(\zeta_{4p}))$. Since $\sigma(A_p)=A_p$ for all $\sigma\in\Gal(\Q(\zeta_{p})/\Q)$, we can define an equivalence relation $\sim$ on $\{x\in\F_q^\times : \Kl_{n,q}(x)\in A_p\}$ generated by $x\sim c^n x$ for all $c\in\F_p^\times$, $x\in\F_q^\times$, and we have
\begin{eqnarray*}
  |\{x\in\F_q^\times : \Kl_{n,q}(x)\in A_p\}/\sim|&=&\frac{|\{x\in\F_q^\times : \Kl_{n,q}(x)\in A_p\}|(p-1,n)}{p-1}\\
  &\ll_n&\frac{|\{x\in\F_q^\times : \Kl_{n,q}(x)\in A_p\}|}{p-1}.
\end{eqnarray*}
If in addition the hypotheses of Proposition \ref{prop:definableSubsetsKl} are satisfied, this yields
\[|\{x\in\F_q^\times : \Kl_{n,q}(x)\in A_p\}/\sim|\ll_{n,\varepsilon} \frac{q^{1-1/(2B_n)}\log{q}}{p^{1-\varepsilon}}.\]
\begin{remark}
  The right-hand side can tend to $0$ with $p\to+\infty$ only when $e<\frac{2B_n}{2B_n-1}$. Since $\frac{2B_n}{2B_n-1}\in(1,2)$, this is the case only for $e=1$. Unfortunately, our estimate on the number of prime ideals of bounded norm in $\Q(\zeta_{4p})$ requires to take $e\gg 1$. If it could be extended to $e=1$ (but see Remarks \ref{rem:methodExtende1} \ref{rem:item:methodExtende1}), the above would show that for $p$ large enough, there is no $x\in\F_p^\times$ such that $\Kl_{n,p}(x)\in \varphi(\Q(\zeta_{4p}))$.
\end{remark}

\subsection{Exploiting monodromy over $\C$}\label{subsec:exploitMonodromyC}

As we mentioned in the previous section, determining integral monodromy groups (as required by Definition \ref{def:coherentFamily} \ref{item:coherentFamily2}), say for a subset of valuations of density 1, is usually difficult.

By using some deep results of Larsen and Pink (relying in particular on the classification of finite simple groups in \cite{Lars95}), the following result allows to obtain coherent families from the knowledge of the monodromy groups over $\overline\Q_\ell$, up to passing to a subfamily of density $1$ depending on $p$.

\begin{theorem}\label{thm:monodromyLP}
  Let $E\subset\C$ be a Galois number field with ring of integers $\Oc$ and let $\Lambda$ be a set of valuations on $\Oc$ of natural density $1$. Let $(\Fc_\lambda)_{\lambda\in\Lambda}$ be a compatible system with $\Fc_\lambda$ a sheaf of $\Oc_\lambda$-modules over $\F_p$. We assume that:
  \begin{enumerate}[start=2]
  \item[(2')] There exists $G\in\{\SL_n,\Sp_{2n}\}$ such that for every $\lambda\in\Lambda$, the arithmetic monodromy group of $\Fc_\lambda\otimes\overline\Q_\ell$ is conjugate to $G(\overline\Q_\ell)$.
  \end{enumerate}
  Then there exists a subset $\Lambda_p\subset\Lambda$ of natural density $1$, depending on $p$ and on the family, such that $(\Fc_\lambda)_{\lambda\in\Lambda_p}$ is coherent, with monodromy group structure $G$.
\end{theorem}

After using Theorem \ref{thm:monodromyLP}, we may apply Proposition \ref{prop:largeSieve} with the coherent subfamily $(\Fc_\lambda)_{\lambda\in\Lambda_p}$ to get
\begin{equation}
  \label{eq:largeSieveLP}
  P\big(t(x)\in A\big)\ll \frac{1}{|(\Lambda_p)_L|}\ll_p \frac{1}{|\Lambda_L|},
\end{equation}
when $L=\floor{q^{1/(2B)}}\to+\infty$, with the implied constant depending on $p$ and on the original family.

\subsubsection{Proof of Theorem \ref{thm:monodromyLP}}

  The idea of the argument, based on \cite{LarsPink92} and \cite{Lars95}, is due to Katz and appears partly in \cite[p. 29]{KowLS06}, \cite[p. 7]{KowWeilNumbers06}, \cite[pp. 188--189]{KowLargeSieve08} (however see Remark \ref{rem:LarsenPinkKow} below), and \cite[Section 7]{Katz12}.\\

  To reduce as much as possible to the situation of \cite{LarsPink92} and \cite{Lars95}, we consider the subset $\Lambda_1\subset\Lambda$ corresponding to ideals of degree $1$ over $\Q$, so that $E_\lambda=\Q_\ell$, $\Oc_\lambda=\Z_\ell$ and $\F_\lambda=\F_\ell$ if $\lambda\in\Lambda_1$ is an $\ell$-adic valuation. By \cite[4.7.1]{Jan05}, for any $S\subset\Spec(\Oc)$, the Dirichlet density of $S$ is equal to the Dirichlet density of the elements of $S$ having degree $1$ over $\Q$. In particular, $\Lambda_1$ has Dirichlet density $1$, and actually natural density $1$ by \cite[Corollary 2, p. 248]{Nark04} (for cyclotomic fields, see also the proof of Proposition \ref{prop:chebotarevCyclotomic}).\\

  In the notations of \cite[Section 3]{Lars95} and definitions of \cite[Section 6]{LarsPink92}, we have the compact $F$-group $\pi_{1,p}$ with compatible system of representations
  \[\big(\rho_\lambda=\rho_{\Fc_\lambda}: \pi_{1,p}\to\GL_n(\Oc_\lambda)=\GL_n(\Z_\ell)\big)_{\lambda\in\Lambda_1}\]
  and Frobenius $\Frob_\alpha$ for $\alpha\in\Ac=\{(x,p^n): n\ge 1, x\in\F_{p^n}\}$. Note that $G$ is a simply connected reductive group scheme over $\Z$, and by hypothesis $\rho_\lambda$ is semisimple.

   For every $\lambda\in\Lambda$, we let $G_\lambda=G_{E_\lambda}$, $\Gamma_\lambda=\rho_\lambda(\pi_{1,p})\le G(\Oc_\lambda)$ the integral monodromy group, $\widetilde \Gamma_\lambda:=\Gamma_\lambda\pmod{\lambda}$ its reduction, and
  \[B=\{\lambda\in\Lambda : \widetilde \Gamma_\lambda\lneq G(\F_\lambda)\}\subset\Lambda\]
  the set of valuations where the monodromy group is smaller than expected. We let moreover:
  \begin{itemize}
  \item For every $\alpha\in\Ac$,
    \[\xi(\alpha)\in \Oc_\lambda[T]\le E[T]\]
    the characteristic polynomial of $\rho_{\lambda}(\Frob_{\alpha})$ (which does not depend on $\lambda\in\Lambda$ by hypothesis).
  \item $K\subset\xi(G_{\Q})$ the $\Q$-rational closed subvariety of codimension $\ge 1$ given by \cite[(3.8)]{Lars95}. There exists a constant $C_\alpha> 0$ such that $\xi(\alpha)\pmod{\ell}\not\in K\pmod{\ell}$ if $\ell>C_\alpha$.
  \item $\Ac'\subset\Ac$ the set of the $\alpha\in\Ac$ such that:
    \begin{enumerate}
    \item $\rho_\lambda(\Frob_\alpha)$ is \textit{regular with respect to $\GL_n$} (see \cite[(3.4)]{Lars95}, \cite[(4.5)]{LarsPink92}) for every $\lambda\in\Lambda$.
    \item $\xi(\alpha)\not\in K$.
    \end{enumerate}
    By \cite[(3.11)]{Lars95}, $\{\Frob_\alpha : \alpha\in\Ac'\}\subset\pi_{1,p}$ is still dense and by \cite[(4.7)]{LarsPink92}:
    \begin{enumerate}
    \item $\rho_\lambda(\Frob_\alpha)$ lies in a unique maximal torus of $T_{\lambda,\alpha}$ of $G_{E_\lambda}$.
    \item $\xi(\alpha)$ is associated to a torus $T_\alpha$ in $\GL_{n,E}$, unique up to $\GL_n(E)$-conjugacy, such that $T_\alpha\times_E E_\lambda$ is conjugate to $T_{\lambda,\alpha}$.
    \item The splitting field of these tori is equal to the splitting field $L_\alpha$ of $\xi(\alpha)$ over $E$ \cite[(4.4)]{LarsPink92}.
    \end{enumerate}
  \item $C'_\alpha\ge 1$ such that $L_\alpha/\Q$ is unramified at any $\ell>C_\alpha'$.
  \item $L$ the intersection of the $L_\alpha$ for $\alpha\in\Ac'$, so that $\Q\subset E\subset L\subset L_\alpha$.
  \end{itemize}
  
  We decompose
  \[B=\left(B\cap(\Lambda\backslash\Lambda_1)\right) \ \bigcup \bigcup_{x\in\Gal(L/E)^\sharp} B_x,\]
  where $B_x=\{\lambda\in\Lambda_1\cap B : [\lambda, L/E]=x\}$.
  
  The upper natural density of $B$ is
  \begin{eqnarray*}
    \overline\delta(B)&=&\limsup_{S\to+\infty} \frac{|\{\lambda\in B : N(\lambda)\le S\}|}{|\{\lambda\in\Lambda : N(\lambda)\le S\}|}\\
    &\le&\overline\delta(\Lambda\backslash\Lambda_1)+\sum_{x\in\Gal(L/E)^\sharp} \overline\delta\left(B_x\right)=\sum_{x\in\Gal(L/E)^\sharp} \overline\delta\left(B_x\right).
  \end{eqnarray*}
  
  Let us fix a class $x\in\Gal(L/E)^\sharp$ and an $\ell'$-adic valuation $\lambda'\in\Lambda_1$ with Frobenius $[\lambda', L/E]=x$.

  If $\lambda\in B_x$, then $\Gamma_\lambda$ is a proper subgroup of $G(\F_\lambda)=G(\F_\ell)$. By \cite[(1.1), (1.19)]{Lars95}, when $\ell\gg_G 1$, every maximal subgroup of $G(\F_\ell)$ is of the form $H(\F_\ell)$, for $H\subset G_{\Z_\ell}$ a smooth $\Z_\ell$-subgroup scheme. By \cite[(3.17)]{Lars95} (see also \cite[(3.8)]{Lars95}), it follows that there exists a maximal proper reductive $\Q_\ell$-subgroup $N$ of $G_\lambda$ (containing a Levi component of $H_{\Q_\ell}$) such that
\[\FM(\lambda,\alpha)\in\FM_{N^\circ}\subsetneq\FM_{G_\lambda}\]
for every $\alpha\in\Ac'$ such that $\ell>D_\alpha=\max(C_\alpha, C'_\alpha)$, where:
\begin{itemize}
\item $N^\circ$ is the identity component of $N$.
\item $\FM(\lambda,\alpha)$ is the isomorphism class of the Frobenius module (i.e. free $\Z$-module of finite rank with an endomorphism of finite order) arising from the character group of the maximal torus $T_{\lambda,\alpha}\le G_\lambda$ containing $\rho_\lambda(F_\alpha)$, with the action of $\Gal(\overline\Q_\ell/\Q_\ell)=\Gal(\overline\Q_\ell/E_\lambda)$. By \cite[(3.14)]{Lars95}, this depends only on $[\ell, L_\alpha/\Q]=[\lambda,L_\alpha/E]$ up to isomorphism.
\item $\FM_{G_\lambda}$ and $\FM_{N^\circ}$ are the set of isomorphism classes of Frobenius modules arising from unramified tori of $G_\lambda$, resp. $N^\circ$.
\end{itemize}

Let $M\in\FM_{G_\lambda}\backslash\FM_{N^\circ}$. As in \cite[(3.15)]{Lars95}, and \cite[(8.2)]{LarsPink92}, we will show that
\begin{center}
  $(\star)$ For every $R\ge 1$, there exist $\alpha_1,\dots,\alpha_R\in\Ac'$ such that $[M]=\FM(\lambda',\alpha_i)$ with $L_{\alpha_i}$ linearly disjoint\footnote{Here, this means that for any $2\le i\le R$, $L_{\alpha_1}\dots L_{\alpha_{i-1}}$ and $L_{\alpha_i}$ are linearly disjoint over $L$, i.e. their intersection is equal to $L$.}.
\end{center}

Assuming this, it follows that if $\ell>\max_{1\le i\le R}{D_{\alpha_i}}$, then for $1\le i\le R$,
\[[\lambda, L_{\alpha_i}/E]=[\ell, L_{\alpha_i}/\Q]\neq [\ell', L_{\alpha_i}/\Q]=[\lambda', L_{\alpha_i}/E]\]
in $\Gal(L_{\alpha_i}/\Q)\ge\Gal(L_{\alpha_i}/E)$, since $M\neq\FM(\lambda,\alpha_i)$. Therefore, by Chebotarev's theorem,
\begin{eqnarray*}
  \overline\delta \left(B_x\right)&\le&\overline\delta\left(\{\lambda\in\Lambda : [\lambda, L_{\alpha_i}/E]\neq[\lambda',L_{\alpha_i}/E] \ \text{for }1\le i\le R\}\right)\\
                                  &=&\left(1-\frac{1}{n!}\right)^R \frac{|x|}{|\Gal(L/E)|}
\end{eqnarray*}
since $[L_{\alpha_i}:E]\le n!$ and by linear disjointedness. Hence $\overline\delta(B)\le (1-1/n!)^R$ for every $R\ge 1$, so that $B$ has natural density $0$ by taking $R\to+\infty$.\\

We now prove $(\star)$. It suffices to show that for any finite Galois extension $F/L$, there exists $\alpha\in\Ac'$ such that $[M]=\FM(\lambda',\alpha)$ with $L_\alpha$ and $F$ linearly disjoint over $L$. We proceed as in \cite[(8.2)]{LarsPink92} (where $E=\Q$).

For $K_1,\dots,K_m$ the intermediate fields of $F/L$ normal over $L$ and minimal with respect to inclusion with this property, we have that $L_\alpha$ is linearly disjoint with $F$ over $L$ if and only if $K_i\not\subset L_\alpha$ for all $1\le i\le m$. This holds in particular if for every $i$ there exists $\lambda_i\in\Lambda_1$ corresponding to a prime that splits in $L_\alpha$, but not in $K_i$.

For every $1\le i\le m$, let $\beta_i\in\Ac'$ be such that $K_{i}\not\subset E_{\beta_i}$. By minimality of $K_i$, we have $E_{\beta_i}\cap K_{i}=L$, so that $\Gal(L_{\beta_i}/L)\times\Gal(K_i/L)$ is contained in
\[\{(\sigma_1,\sigma_2)\in\Gal(L_{\beta_i}/E)\times\Gal(K_i/E) : \sigma_1\mid_L=\sigma_2\mid_L\}\cong\Gal(L_{\beta_i}K_i/E).\]
By Chebotarev's theorem, the set of $\lambda\in\Spec(\Oc)$ that split in $L_{\beta_i}$ but does not split in $K_i$ has positive Dirichlet density, so the same holds for the $\lambda\in\Lambda_1$ with this property, since $\Lambda_1$ has Dirichlet density $1$. Hence, there exists $\lambda_i\in\Lambda_1\backslash\{\lambda'\}$ that splits in $L_\alpha$ but not in $K_i$, and we may suppose all the $\lambda_i$ distinct.

By \cite[(7.5.3)]{LarsPink92}, there exists $\alpha\in\Ac'$ such that:
\begin{enumerate}
\item $T_{\lambda',\alpha}$ is conjugate in $\GL_n(E_{\lambda'})$ to the unramified maximal torus of $G_{\lambda'}$ corresponding to $M$, so $[M]=\FM(\lambda',\alpha)$.
\item $T_{\lambda_i,\alpha}$ is conjugate in $\GL_n(E_{\lambda_i})$ to $T_{\lambda_i,\beta_i}$. Since $\lambda_i$ splits in $L_{\beta_i}$, this torus is split, so that $\lambda_i$ also splits in $L_\alpha$.
\end{enumerate}

This concludes the argument.

  
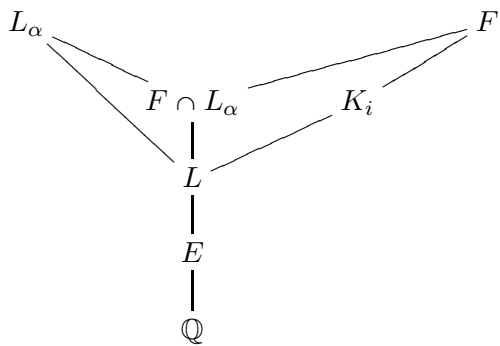
\begin{equation*}
    \xymatrix@R=0.5cm@C=1.1cm{
      L_\alpha\ar@{-}[ddr]\ar@{-}[dr]&&&F\ar@{-}[dl]\ar@{-}[dll]\\
      &F\cap L_\alpha\ar@{-}[d]&K_i\ar@{-}[dl]&\\
      &L\ar@{-}[d]&&\\
      &E\ar@{-}[d]&&\\
      &\Q&&
    }
  \end{equation*}

Finally, concerning the geometric integral monodromy group $\Gamma^\geom_\lambda=\rho_\lambda(\pi_{1,p}^\geom)\normal\Gamma_\lambda$, note that:
  \begin{enumerate}
  \item\label{item:proof:monodromyLP1} $\widetilde \Gamma_\lambda/\widetilde \Gamma^\geom_\lambda$ is a finite quotient of $\pi_{1,p}/\pi_{1,p}^\geom\cong\widehat\Z$, hence a finite cyclic group.
  \item\label{item:proof:monodromyLP2} If $|\F_\lambda|\gg_G 1$, the group $G'(\F_\lambda):=G(\F_\lambda)/Z(G(\F_\lambda))$ is simple nonabelian (see e.g. \cite[Theorem 24.17]{TesMal11}).
  \end{enumerate}
  Hence, by \ref{item:proof:monodromyLP2}, if $\widetilde \Gamma^\geom_\lambda\lneq G(\F_\lambda)$, then it is contained in $Z(G(\F_\lambda))$, so that
  \[G'(\F_\lambda)\cong \frac{\widetilde \Gamma_\lambda/\widetilde \Gamma^\geom_\lambda}{Z(G(\F_\lambda))/\widetilde \Gamma^\geom_\lambda}\]
  would be cyclic by \ref{item:proof:monodromyLP1}, a contradiction. \qed

\begin{remarks}\label{rem:LP}
  \begin{enumerate}
  % \item   If $\Fc_\lambda$ is pointwise pure of weight $0$, we see as in \cite[9.2.4]{KatzSarnak91} that if $C$ is a maximal compact subgroup of $G(\overline\Q_\ell)$, then for every $\alpha\in\Ac$, the semisimple part $\rho(\Frob_\alpha)^\semisimple$ gives a well-defined conjugacy class in $C$ with minimal polynomial $\xi(\alpha)$. If $K$ is defined by the polynomials $f_1,\dots,f_m\in\Q[X]$, then we can take uniformly
  %   \[C_\alpha=\max_{\alpha'\in\Ac'}\max_{1\le i\le m} |f_i(\xi(\alpha'))|\]
  %   for all $\alpha\in\Ac$.
  \item We consider compatible systems of representations $\rho_\lambda: \pi\to\GL_n(\Oc_\lambda)$, where $\lambda$ is a valuation on the ring of integers $\Oc$ of a number field $E/\Q$, while the results in \cite[Part II]{LarsPink92} and \cite{Lars95} are stated for the case $E=\Q$. One needs to be cautious before stating the natural generalizations of the results of Larsen and Pink. For example, under the notations of the theorem, the maximal subgroups of $G(\F_\lambda)$ are not all of the form $H(\F_\lambda)$ for $H\subset G_{\Oc_\lambda}$ a smooth $\Oc_\lambda$-subgroup scheme, unless $\F_\lambda=\F_\ell$ as in \cite[(1.1), (1.19)]{Lars95}: for instance, one has subfield subgroups.
  \item\label{rem:LPSO} Theorem \ref{thm:monodromyLP} cannot be used when $G=\SO_n$, since it is not simply connected, and this assumption is required for \cite[(1.19)]{Lars95}. In even dimension, note that one would need additional input to determine the type ($+$ or $-$) of the monodromy groups over $\F_\lambda$.
\end{enumerate}
\end{remarks}

\subsubsection{Arithmetic and geometric monodromy groups}\label{subsubsec:twisting}

  Often, only the geometric monodromy group is determined, while Theorem \ref{thm:monodromyLP} and Definition \ref{def:coherentFamily} require knowledge of the arithmetic monodromy. By twisting a sheaf $\Fc_\lambda$ by a constant or a Tate twist, it is often possible to get a sheaf $\Fc'_\lambda$ with
  \[G_\geom(\Fc_\lambda)=G_\geom(\Fc'_\lambda)\le G_\arith(\Fc'_\lambda)\le G_\geom(\Fc'_\lambda),\]
  so that $G_\geom(\Fc'_\lambda)=G_\arith(\Fc'_\lambda)=G_\geom(\Fc_\lambda)$. Examples will be given in the next sections.

\begin{remark}\label{rem:LarsenPinkKow}
  In \cite{KowLS06,KowWeilNumbers06,KowLargeSieve08}, the results of Larsen--Pink are applied to deduce the geometric monodromy group over $\F_\ell$ from the geometric group over $\overline\Q_\ell$ . However, this is incorrect since the geometric group does not contain a dense subset of the Frobenius. Moreover, note that the arithmetic monodromy group is not contained in $\Sp_{2g}(\overline\Q_\ell)$ (but in $\GSp_{2g}(\overline\Q_\ell)$).

    For the unnormalized family of first cohomology groups of hyperelliptic curves, this is not an issue because the results of J.-K. Yu and C. Hall also apply to give the geometric monodromy groups. Alternatively, one may normalize by a Tate twist as in Proposition \ref{prop:pointCounting} and apply Theorem \ref{thm:monodromyLP} to the normalized sheaf (see above).

    For the characteristic 2 example of \cite[Proposition 3.3]{KowWeilNumbers06}, the result of Hall can also be applied because the local monodromy at $0$ is a unipotent pseudoreflection. Again, one could also apply Theorem \ref{thm:monodromyLP} after normalizing.

    On the other hand, the statement \cite[Theorem 6.1]{KowLS06} must be modified to assume for example that the \textit{arithmetic} monodromy group is $\Sp$, or that the geometric monodromy groups over $\F_\ell$ are known for all $\ell\gg 1$.
\end{remark}

\subsection{General exponential sums}

Finally, we use the previous section to give examples of coherent families of the form \eqref{eq:expsumGen}.

\subsubsection{Construction of coherent families}

\begin{proposition}[Exponential sums \eqref{eq:expsumGen}, $h=0$, $\chi=1$]\label{prop:expsumGen1}
  Let $f\in\Q(X)$ and let $Z_{f'}$ be the set of zeros of $f'$ in $\C$, having cardinality $k_f$. We assume that the zeros of $f'$ are simple, that $|f(Z_{f'})|=|Z_{f'}|$ (i.e. $f$ is \emph{supermorse}), and that either:
  \begin{itemize}
  \item $(H_1)$: $k_f$ is even, $\sum_{z\in Z_{f'}} f(z)=0$, and if $s_1-s_2=s_3-s_4$ with $s_i\in f(Z_{f'})$, then $s_1=s_3,s_2=s_4$ or $s_1=s_2,s_3=s_4$.
  \item $(H_2)$: $f$ is odd, and if $s_1-s_2=s_3-s_4$ with $s_i\in f(Z_{f'})$, then $s_1=s_3,s_2=s_4$ or $s_1=s_2,s_3=s_4$ or $s_1=-s_4,s_2=-s_3$.
  \end{itemize}
  If $p$ is large enough, for $E=\Q(\zeta_{4p})$ and $\Lambda_{0,p}$ as in Example \ref{ex:Lambda}\ref{item:exLambdaI}, there exists a family $(\Gc_{f,\lambda})_{\lambda\in\Lambda_{0,p}}$ of sheaves of $\Oc_\lambda$-modules over $\F_p$, with trace function
  \[x\mapsto \frac{-1}{\sqrt{q}}\sum_{\substack{y\in\F_q\\f(y)\neq\infty}}e\left(\frac{\tr(xf(y))}{p}\right)\hspace{0.2cm} (x\in\F_q),\]
  and conductor depending only on $f$.
  
  Moreover, there exists $\alpha_p\in\overline\Q$ and a set of valuations $\Lambda'=\Lambda'_{f,p}$ of density $1$ on $E'=E(\alpha_p)$, depending only on $f$ and $p$, such that
  \[\left(\Gc_{f,\lambda}\otimes\alpha_p\Oc'_\lambda\right)_{\lambda\in\Lambda'}\]
    is a coherent family of sheaves of $\Oc_\lambda'$-modules over $\F_p$, for $\Oc'$ the ring of integers of $E'$, with monodromy group structure
    \begin{itemize}
    \item $G=\SL_{k_f}$ if $(H_1)$ holds.
    \item $G=\Sp_{k_f}$ if $(H_2)$ holds, and one may take $\alpha_p=1$.
    \end{itemize}
\end{proposition}
\begin{proof}
  See \cite[Theorem 7.9.4, Lemmas 7.10.2.1, 7.10.2.3]{KatzESDE} for the construction and \cite[7.9.6--7, 7.10]{KatzESDE} for the determination of $G_\geom^\circ(\Gc_{f,\lambda})$ over $\C$. The family forms a compatible system by Lemma \ref{lemma:compSysFT}. The definition over $\Oc_\lambda$ comes from the definition of the $\ell$-adic Fourier transform on the level of sheaves of $\Oc_\lambda$-modules (see \cite[Chapter 5]{KatzGKM}). Under our hypotheses, $G_\geom(\Gc_{f,\lambda})$ contains $\SL_{k_f}(\overline\Q_\ell)$, resp. $\Sp_{k_f}(\overline\Q_\ell)$. Moreover:
  \begin{itemize}
  \item In the $(H_2)$ case, $G_\arith(\Gc_{f,\lambda})\le\Sp_{k_f}(\F_\lambda)$ by \cite[7.10.4 (3)]{KatzESDE}, and we can apply Theorem \ref{thm:monodromyLP}.
  \item In the $(H_1)$ case, since $\pi_{1,p}/\pi_{1,p}^\geom\cong\hat\Z$ is abelian, there exists by Clifford theory an element $\beta_p\in\Oc_\lambda^\times\cap\overline\Q$ not depending on $\lambda$ (since we have a compatible system) such that the  determinant is isomorphic to $\beta_p\otimes\Oc_\lambda$.

    As in Section \ref{subsubsec:twisting}, we obtain that with $\alpha_p=\beta_p^{-1/k_f}\in\Oc'_{\lambda'}$ for any valuation $\lambda'$ of $\Oc'$ extending $\lambda$, the arithmetic and geometric monodromy groups of $\Gc_{f,\lambda'}\otimes \alpha_p\Oc'_{\lambda'}$ coincide and are conjugate to $\SL_{k_f}(\overline\Q_\ell)$, so that we can apply Theorem \ref{thm:monodromyLP}.
  \end{itemize}
\end{proof}
\begin{example}
  The hypotheses hold for the rational function $f=aX^{r+1}+bX$, where $a,b\in\Z$, $r\in\Z-\{1\}$, $rab\neq 0$, with $(H_1)$ if $r$ is odd and $(H_2)$ otherwise (see \cite[p. 7]{FouvryMichelSommes}), or for the polynomial $f=X^n-naX$, where $a\in\Z\backslash\{0\}$ and $n\ge 3$, with $(H_1)$ if $n$ is even, $(H_2)$ otherwise.
\end{example}

The following include for example Birch sums (with $h=X^3$):
\begin{proposition}[Exponential sums \eqref{eq:expsumGen}, $f=X$, $\chi=1$, $h$ polynomial]\label{prop:expsumGen2}
  Let $h=\sum_{i=1}^n a_iX^i\in\Z[X]$ be a polynomial of degree $n\ge 3$ with $n\neq 7,9$ and $a_{n-1}=0$. If $p$ is large enough, for $E=\Q(\zeta_{4p})$ and $\Lambda_{0,p}$ as in Example \ref{ex:Lambda}\ref{item:exLambdaI}, there exists a family $(\Gc_{h,\lambda})_{\lambda\in\Lambda_{0,p}}$ of sheaves of $\Oc_\lambda$-modules over $\F_p$ with trace function
  \[x\mapsto\frac{-1}{\sqrt{q}}\sum_{y\in\F_q} e\left(\frac{\tr(xy+h(y))}{p}\right) \ (x\in\F_q),\]
  and conductor depending only on $h$.

  Moreover, there exists $\alpha_p\in\overline\Q$ and a set of valuations $\Lambda'=\Lambda'_{h,p}$ of density $1$ on $E'=E(\alpha_p)$, depending only on $h$ and $p$, such that
  \[\left(\Gc_{h,\lambda}\otimes \alpha_p\Oc'_\lambda\right)_{\lambda\in\Lambda'}\]
  is a coherent family of sheaves of $\Oc_\lambda'$-modules over $\F_p$, for $\Oc'$ the ring of integers of $E'$, with monodromy group structure:
  \begin{enumerate}
  \item $G=\Sp_{n-1}$ if $n$ is odd and $h$ has no monomial of even positive degree; one may take $\alpha_p=1$.
  \item $G=\SL_{n-1}$ otherwise.
\end{enumerate}
\end{proposition}
\begin{proof}
  This is similar to the proof of Proposition \ref{prop:expsumGen1}. See \cite[7.12]{KatzESDE} for the construction of the sheaves and the determination of the monodromy groups over $\C$. In the symplectic case, ibidem shows that the arithmetic monodromy group is itself contained in $\Sp_{n-1}$.
\end{proof}

\begin{proposition}[Exponential sums \eqref{eq:expsumGen}, $f$ polynomial, $\chi\neq 1$]\label{prop:expsumGen3}
    Let
  \begin{itemize}
  \item $h\in\Q(X)$ odd with a pole of order $n\ge 1$ at $\infty$.
  \item $f\in\Z[X]$ odd nonzero of degree $d$ with $(d,n)=1$.
  \item $\chi$ a character of $\F_p^\times$ of order $r\ge 2$.
  \item $g\in\Q(X)$ nonzero, with the order of any zero or pole not divisible by $r$.
  \end{itemize}
  For $p$ large enough, for $E=\Q(\zeta_{4p})$ and $\Lambda_{0,p}$ as in Example \ref{ex:Lambda}\ref{item:exLambdaI}, there exists a family $(\Gc_{h,f,\chi,g,\lambda})_{\lambda\in\Lambda_{0,p}}$ of sheaves of $\Oc_\lambda$-modules over $\F_p$ with trace function \eqref{eq:expsumGen} and conductor depending only on $f,g,h,r$.
  %, with $\Lambda_{0,p}$ as defined in Example \ref{ex:Lambda}\ref{item:exLambdaI}.

  Moreover, if we assume that there exists $L\in\Q(X)$ even with $L(x)^r=g(x)g(-x)$ and either $N=\rank(\Gc_{h,f,\chi,g,\lambda})\neq 8$ or $|n-d|\neq 6$, then there exists a set of valuations $\Lambda_{p}\subset\Lambda_{0,p}$ of density $1$, depending only on $h,f,g,\chi$ and $p$, such that
  \[\left(\Gc_{h,f,\chi,g,\lambda}\right)_{\lambda\in\Lambda_p}\] is a coherent family, with monodromy group structure $G=\Sp_N$.
  %, where $\Lambda_{0,p}$ is as in Example \ref{ex:Lambda}\ref{item:exLambdaI}.
\end{proposition}
\begin{proof}
  This is again similar to the proof of Proposition \ref{prop:expsumGen1}. See \cite[7.7, 7.13 ($\Sp$-example(2))]{KatzESDE} for the construction of the sheaves and the determination of the monodromy groups over $\C$; \cite[7.13]{KatzESDE} shows that the arithmetic monodromy group is itself contained in $\Sp_{N}$.
\end{proof}
\begin{remark}
  If $L$ as in the statement of Proposition \ref{prop:expsumGen3} is odd, there exists $\alpha_p\in\{\pm 1\}$ such that the arithmetic and geometric monodromy groups over $\C$ of $\alpha_p\otimes\Gc_{h,f,\chi,g,\lambda}$ coincide and are conjugate to $\SO_N(\C)$ (see \cite[7.14 ($\O$-example(2))]{KatzESDE}). However, Theorem \ref{thm:monodromyLP} does not apply in that case (see Remarks \ref{rem:LP} \ref{rem:LPSO}).
\end{remark}

\subsubsection{Zero-density estimates}

Hence, for the three families above, we get by Corollary \ref{cor:largeSieveExplicitQzeta4p} with Propositions \ref{prop:powersFiniteIndex} and \ref{prop:definableSubsets}:
\begin{proposition}\label{prop:largeSieveExpSums}
  We fix a prime $p$ and we set $q=p^e$. Let $t: \F_q\to \Q(\zeta_{4p})$ be the trace function associated with one of the families from Propositions \ref{prop:expsumGen1}, \ref{prop:expsumGen2} or \ref{prop:expsumGen3}, and let $B$ be as in \eqref{eq:BGeneral}.

  For $\varphi(x)$ a first-order formula in the language of rings as in Proposition \ref{prop:definableSubsets},
  \[P\big(t(x)\in \varphi(\Q(\zeta_{4p}))\big)\ll_{p,f,\varphi} \frac{\log{q}}{B q^{\frac{1}{2B}}}\to 0 \ (e\to+\infty).\]
  In particular, for almost all monic $f\in\Z[X]$ of fixed degree $\ge 2$ (such as $f(X)=X^m$ for $m\ge 2$ coprime to $p$),
  \[P\big(t(x)\in f(\Q(\zeta_{4p}))\big)\ll_{p,f} \frac{\log{q}}{B q^{\frac{1}{2B}}}\to 0 \ (e\to+\infty).\]
\end{proposition}
\begin{proof}
  In the symplectic case, this is immediate. In the special linear cases, we get the result for the twisted trace function $t': \F_q\to\Oc_\lambda'$, $t'(x)=\alpha_p^et(x)$. The result for the unnormalized function is obtained as in Section \ref{subsubsec:unnormalized}, replacing $A$ by $\alpha_p^{-e}A$ in Proposition \ref{prop:largeSieve} and using uniformity.
\end{proof}
\begin{remark}
  In the special linear case, the implied constant depends on $p$ both because of the use of Theorem \ref{thm:monodromyLP}, and because of the twisting factor $\alpha_p$.
\end{remark}

\subsubsection{Galois actions}

Note that for the sums
\[\frac{-1}{\sqrt{q}}\sum_{y\in\F_q} e\left(\frac{\tr(xf(y)+h(y))}{p}\right)\chi(g(y)) \hspace{0.2cm} (x\in\F_q^\times)\]
with $h(Y)=Y^m$ and $f(Y)=Y^n$ $(m,n\in\Z)$, we have $\sigma_{c^m}(t(x)))=t(c^{m-n}x)$, where $\sigma_{c^m}\in\Gal(\Q(\zeta_p)/\Q)\cong\F_p^\times$ corresponds to $c^m$ for some $c\in\F_p^\times$. Hence, as in Section \ref{subsubsec:GaloisKl}, it makes sense to study the \textit{integer}
\[\frac{|\{x\in\F_q^\times : t(x)\in \varphi(\Q(\zeta_{4p}))\}|(p-1,m-n)}{p-1}\ll_{m,n}\frac{|\{x\in\F_q^\times : t(x)\in \varphi(\Q(\zeta_{4p}))\}|}{p-1}\]
when $\varphi(x)$ is a first-order formula in the language of rings. However, doing so requires an estimate of the form \eqref{eq:density} uniform in $p$, for example through a more precise knowledge of the integral monodromy instead of relying on Theorem \ref{thm:monodromyLP}.

\subsection{Hypergeometric sums}

The same methods also apply to the hypergeometric sums defined by Katz \cite[Chapter 8]{KatzESDE}, generalizing Kloosterman sums: under some conditions, the arithmetic and geometric monodromy groups over $\overline\Q_\ell$ coincide and are conjugate to $\SL_n$, without needing to twist (see the references to \cite{KatzESDE} in \cite[Proposition 7.7]{PGGaussDistr16}).

\bibliographystyle{alpha}
\small
\bibliography{references}
\normalsize

\end{document}